\documentclass[twoside]{article}

\usepackage[accepted]{aistats2025}
\usepackage[round]{natbib}

\usepackage{graphicx} % support the \includegraphics command and options
\graphicspath{{./Figs/}}

\usepackage[skip=0.333\baselineskip]{caption}
\usepackage{subcaption}
\usepackage{lipsum}
\usepackage[htt]{hyphenat}
\usepackage{url}
\usepackage{enumitem}

\usepackage{amsmath}
\usepackage{amssymb}
\usepackage{mathtools}
\usepackage{amsthm}
\usepackage{algorithm}
\usepackage{algorithmic}

\usepackage{csquotes}
\usepackage{comment}

\usepackage[capitalize,noabbrev]{cleveref}
\crefname{problem}{Problem}{Problems}
\crefformat{problem}{Problem~(#2#1#3)}
\crefrangelabelformat{Problem}{#3(#1)#4 to #5(#2)#6}
\crefname{assumption}{Assumption}{Assumptions}
\usepackage{color}

\theoremstyle{plain}
\newtheorem{theorem}{Theorem}[section]
\newtheorem{proposition}[theorem]{Proposition}
\newtheorem{lemma}[theorem]{Lemma}

\theoremstyle{definition}
\newtheorem{definition}[theorem]{Definition}
\newtheorem{assumption}[theorem]{Assumption}
\theoremstyle{remark}

\newcommand{\bbR}{\mathbb{R}}

\newcommand{\bbE}{\mathbb{E}}
\newcommand{\bbP}{\mathbb{P}}

\newcommand{\blambda}{\boldsymbol{\lambda}}

\newcommand{\lambdarv}{\tilde{\lambda}}

\newcommand{\betabar}{\boldsymbol{\bar{\beta}}}
\newcommand{\bbeta}{\boldsymbol{\beta}}
\newcommand{\btheta}{\boldsymbol{\theta}}
\newcommand{\thetahat}{\boldsymbol{\hat{\theta}}}
\newcommand{\bPhi}{\boldsymbol{\Phi}}

\newcommand{\betaavg}{\boldsymbol{\beta}_\mathrm{avg}^\ast}
\newcommand{\betasp}{\boldsymbol{\beta}_{\rm sp}^\ast}
\newcommand{\epsilonsp}{\epsilon_{\rm sp}}

\newcommand{\bv}{\boldsymbol{v}}
\newcommand{\bw}{\boldsymbol{w}}
\newcommand{\bx}{\boldsymbol{x}}

\newcommand{\abs}[1]{\left|#1\right|}
\newcommand{\Ib}[1]{\mathbb I \left\{ #1 \right\} }
\newcommand{\Eb}[1]{\mathbb E \left[ #1 \right]}

\newcommand{\lsim}{\lesssim}

\newcommand{\blambdarv}{\bm{\tilde \lambda}}
\usepackage{bm}

\begin{document}



\twocolumn[

\aistatstitle{Beyond Discretization: Learning the Optimal Solution Path}

\aistatsauthor{ Qiran Dong \And Paul Grigas \And Vishal Gupta }

\aistatsaddress{ University of California, Berkeley \And  University of California, Berkeley \And University of Southern California } ]

\begin{abstract}
  Many applications require minimizing a family of optimization problems indexed by some hyperparameter $\lambda \in \Lambda$
  to obtain an entire solution path.  
 Traditional approaches proceed by discretizing $\Lambda$ and solving a series of optimization problems.
  We propose an alternative approach
  that parameterizes the solution path with a set of basis functions and solves a \emph{single} stochastic optimization problem to learn the entire solution path.  Our method 
  offers substantial complexity improvements over discretization. When using constant-step size SGD, the uniform error of our learned solution path relative to the true path exhibits linear convergence to a constant related to the expressiveness of the basis. When the true solution path lies in the span of the basis, this constant is zero. We also prove stronger results for special cases common in machine learning: When $\lambda \in [-1, 1]$ and the solution path is $\nu$-times differentiable, constant step-size SGD learns a path with $\epsilon$ uniform error after at most $O(\epsilon^{\frac{1}{1-\nu}} \log(1/\epsilon))$ iterations, and when the solution path is analytic, it only requires $O\left(\log^2(1/\epsilon)\log\log(1/\epsilon)\right)$.  By comparison, the best-known discretization schemes in these settings require at least $O(\epsilon^{-1/2})$ discretization points (and even more gradient calls). Finally, we propose an adaptive variant of our method that sequentially adds basis functions and demonstrate strong numerical performance through experiments.
\end{abstract}

%!TEX root = 0_Main.tex

\section{INTRODUCTION}

Many decision-making applications entail solving a family of \emph{parametrized} optimization problems:
\begin{equation}\label[problem]{eq:ParametricProb}
    \btheta^*(\lambda) \in \arg
     \min_{\btheta\in \bbR^d} h(\btheta, \lambda),\quad \lambda\in\Lambda,
\end{equation} 
where $\Lambda$ is an arbitrary set of parameters indexing the problems. 
% and the optimization admits an optimal solution $\btheta^*(\lambda)$ for each $\lambda \in \Lambda$. 
(We assume \cref{eq:ParametricProb} admits an optimal solution for each $\lambda \in \Lambda$.)  In such applications, we often seek to compute the entire solution path 
$\{\btheta^*(\lambda): \lambda \in \Lambda\}$ in order to present experts with a portfolio of possible solutions to compare and assess tradeoffs.

% one seeks the ``best" hyperparameter $\lambda^*$, i.e., one such that $\btheta^*(\lambda^*)$ minimizes some auxiliary criterion, e.g., out-of-sample performance.  By contrast, when there is no obvious auxiliary criterion, we might prefer to compute the
% entire solution path $\{\btheta^*(\lambda): \lambda \in \Lambda\}$ and present experts with a ``portfolio" of possible solutions which they can compare qualitatively to assess tradeoffs.  In this paper, we focus this second question of  computing this entire solution path efficiently.

% In these settings, we typically seek the entire solution path $\{\btheta^*(\lambda): \lambda \in \Lambda\}$ for a variety of reasons, e.g., so that experts might compare solutions qualitatively to assess tradeoffs.
%In other words, we need to solve $P(\lambda)$ for all $\lambda \in \Lambda$. 

%\qrd{@Paul: What's the written comment saying here?}\paul{I wonder if we should use the phrase ``portfolio'' (I think I'm borrowing this from Swati) or something similar to sell the idea that decision-makers often need to be presented with a whole portfolio of potential solutions.} 
As an example, consider the $p$-norm fair facility location problem \citep{gupta2023lp}, which minimizes facility opening costs while incorporating $l_p$-regularization to promote fairness across socioeconomic groups. Since there is no obvious choice for $p$, we might prefer computing solutions for all $p$ and allowing experts to (qualitatively) assess the resulting solutions.  %Similar issues often arise in other fair resource allocation problems.
%
%
% As a second example, consider
% binary classification with imbalanced classes \citep{johnson2019survey}. A standard approach is to upweight the minority class, but it is not apriori obvious what the right weight should be.  Different choices induce different tradeoffs between Type I and Type II error and the real-world costs associated with these errors may be hard to assess.  Hence, we might prefer to train a set of classifiers across all weightings and compare them across \emph{multiple} performance metrics (precision, explainability, implementation cost, etc.).
Many other applications 
entail navigating similar tradeoffs, including %where 
%there is no clear ``best" answer and hence 
%we would prefer to learn the entire solution path: 
upweighting the minority class in binary classification to tradeoff Type I and II errors 
%\citep{johnson2019survey} 
or
%demographic classes to improve the within-group accuracy at the expense of overall accuracy, 
aggregating features to increase interpretability at the expense of accuracy.
%\citep{yan2021rare}.
In each case, we seek the entire solution path because 
selecting the ``best" solution requires weighing a variety of criteria, some of which may be qualitative.
These settings differ from classical hyperparameter tuning where there is a clear auxiliary criterion (like out-of-sample performance), and it would be enough to identify the single $\lambda^* \in \Lambda$ such that $\btheta^*(\lambda^*)$  optimizes this criteria.

The most common approach to learning the entire solution path is discretization: discretize $\Lambda$, solve \cref{eq:ParametricProb} at each grid point, and interpolate the resulting solutions.  With enough grid points, interpolated solutions are approximately optimal along the entire path.  Several authors have sought to determine the minimal the number of discretization points needed to achieve a target level of accuracy, usually under minimal assumptions on $h(\btheta, \lambda)$.  
%(See \cref{sec:OtherRelatedWork} below for other related work that focuses on specific instances of $h(\cdot, \cdot)$.)  
%\vg{@Qiran: You cite this paper here originally: \citep{ndiaye2017gap}.  I don't see how this paper (2017) is related. Was it a typo?  Did you mean the 2019 paper?}
%The complexity of these discretization methods then depends on the minimum number of partitions required for $\Lambda$, as well as the computational complexity of each subproblem. 
% \qrd{@Paul: this part looks like it's in a right place to me. Maybe I'm missing something?}\paul{Pascal's paper has nothing to do with discretization.} 
\citet{giesen2012approximating} considers convex optimization problems over the unit simplex when $\Lambda \subseteq \mathbb R$ and show that learning the solution path to accuracy $\epsilon$ requires at least $O(1/\epsilon)$ grid points. \citet{giensen2012approximating} consider the case where $h(\btheta, \lambda)$ is concave in $\lambda$ and $\Lambda \subseteq \mathbb R$ and show only $O(1/\sqrt{\epsilon})$ points are needed.   
% proposed the first \emph{absolute} $\epsilon$-approximation of the solution path assuming $P(\lambda)$ is concave in $\lambda \in \Lambda \subseteq \mathbb R$, requiring $O(\epsilon^{-1/2})$ grid points. \vg{@Qiran: What is $P(\lambda)$ in the previous statement?}
More recently, \citet{ndiaye2019safe} relate the required number of points to the regularity of $h(\btheta, \lambda)$, arguing that if $h$ is uniformly convex of order $d$, one requires 
$O(1/\sqrt[d]{\epsilon})$ points. 

A potential criticism of this line of research is that the computational complexity of these methods depends not only on the number of grid points but also depends on the amount of work per grid point, e.g., as measured by the number of gradient calls used by a first-order method.  Generally speaking, these methods do not share much information across grid points, at most warm-starting subsequent optimization problems.  
However, gradient evaluations at nearby grid points contain useful information for optimizing the current grid point, and leveraging this information presents an opportunity to reduce the total work. 

\subsection{Our Contributions}

We propose a novel, simple algorithmic procedure to learn the solution path that applies to an arbitrary set $\Lambda$.
%
% Assume the parametric optimization problem \eqref{eq:ParametricProb} admits an optimal solution $\btheta^*(\lambda)$ for each $\lambda \in \Lambda$.
% Formally define the solution path $\btheta^*(\cdot): \Lambda \mapsto \bbR^d$ by 
% \begin{equation} \label[problem]{eq:ParametricProb}
%     \btheta^*(\lambda) \in \arg\min_{\btheta \in \bbR^d} \ h(\btheta, \lambda), \quad \text { for all } \lambda\in\Lambda,
% \end{equation}
% where ties are broken arbitrarily.
%Assuming an optimal solution of $P(\lambda)$ exists for all $\lambda\in\Lambda$, let $\btheta^*(\cdot): \Lambda \mapsto \bbR^d$ be a function that returns an arbitrary optimal solution, i.e.,
%
%In reality, the exact solution to this problem is often not obtainable within reasonable computation cost\qrd{example}. Instead, we may want to take the second best option available, which is finding an $\epsilon$-solution path, defined as follows.
% \begin{definition}
%     Given $\epsilon > 0$, an \emph{$\epsilon$-solution path} $\thetahat(\cdot): \Lambda \mapsto \bbR^d$ is a function such that 
%     \[
%     h(\thetahat(\lambda), \lambda) - h(\btheta^*(\lambda), \lambda) \leq \epsilon, \quad \text { for all } \lambda\in\Lambda.
%     \]
% \end{definition}
The key idea is to replace the family of problems in \eqref{eq:ParametricProb} with a {\em single} stochastic optimization problem. This stochastic optimization problem depends on two user-specified components: a distribution $\bbP_\lambda$ over values of $\lambda \in \Lambda$ and a collection of basis functions $\bPhi_j(\cdot): \Lambda \to \bbR^{d}$, $j=1, \ldots, p$.  We then seek to approximate $\btheta^*(\lambda)$ as a linear combination of basis functions, $\bPhi_{1:p}(\lambda)\hat{\bbeta}$,  where 
$\bPhi_{1:p}:= \left[\bPhi_1\ \bPhi_2\ \cdots\ \bPhi_p\right]$ and $\hat{\bbeta}$ is an approximate solution to

% $\bPhi_{1:p}(\lambda) \in $ os the matrix whose $j^\text{th}$ column is $\bPhi_j(\lambda)$, i.e. . 

% Specifically, let $\lambdarv$ be a user-defined random variable with distribution $D$ and let $\bPhi_i(\lambda) \in \bbR^d$ for $i =1, \ldots, p$ be user-defined, vector-valued functions of $\lambda$.  Let $\bPhi(\lambda) \in \bbR^{p \times d}$ be the matrix whose $i^\text{th}$ row equals to $\bPhi_i(\lambda)$. 
%Let $\bbeta \in \bbR^p$ be a vector of coefficients on the basis functions $\bPhi_i(\cdot)$. 
% Our goal is to approximate $\btheta^*(\cdot)$ as a linear combination of basis functions. 
% We formalize the error of a given linear combination in Definition \ref{def:sp_error} below.
% \begin{definition}\label{def:sp_error}
% For any $\bbeta \in \bbR^p$, we define the {\em solution path error} of $\bbeta$ as
% \begin{equation*}
% \epsilonsp(\bbeta) := \sup_{\lambda\in \Lambda}\left\{h(\bPhi(\lambda) \bbeta, \lambda) - h(\btheta^*(\lambda), \lambda)\right\}.
% \end{equation*}
% \end{definition}
\begin{equation}\label[problem]{eq:StochasticProb}
    \min_{\bbeta \in \bbR^p} \  \bbE_{\lambdarv\sim\mathbb{P}_\lambda}\left[h(\bPhi_{1:p}(\lambdarv) \bbeta, \lambdarv)\right].
\end{equation}

In contrast to discretization schemes which only leverage local information, (stochastic) gradient evaluations of \cref{eq:StochasticProb} inform \emph{global} structure.
Moreover, through a suitable choice of basis functions, we can naturally accomodate complex sets $\Lambda$
 in contrast to earlier work that only treats $\Lambda \subseteq \mathbb R$.
 %, without suffering the curse of dimensionality \citep{gjm-elbpcrp-12}. 
 Finally, any stochastic optimization routine can be used on \cref{eq:StochasticProb} beyond just SGD (see, e.g., \citet{lan2020first, bottou2018optimization}, among others).
 % including ADAM, trust-region or accelerated methods.

\begin{figure}[hbt!] 
\begin{center}
\includegraphics[width=.35\textwidth]{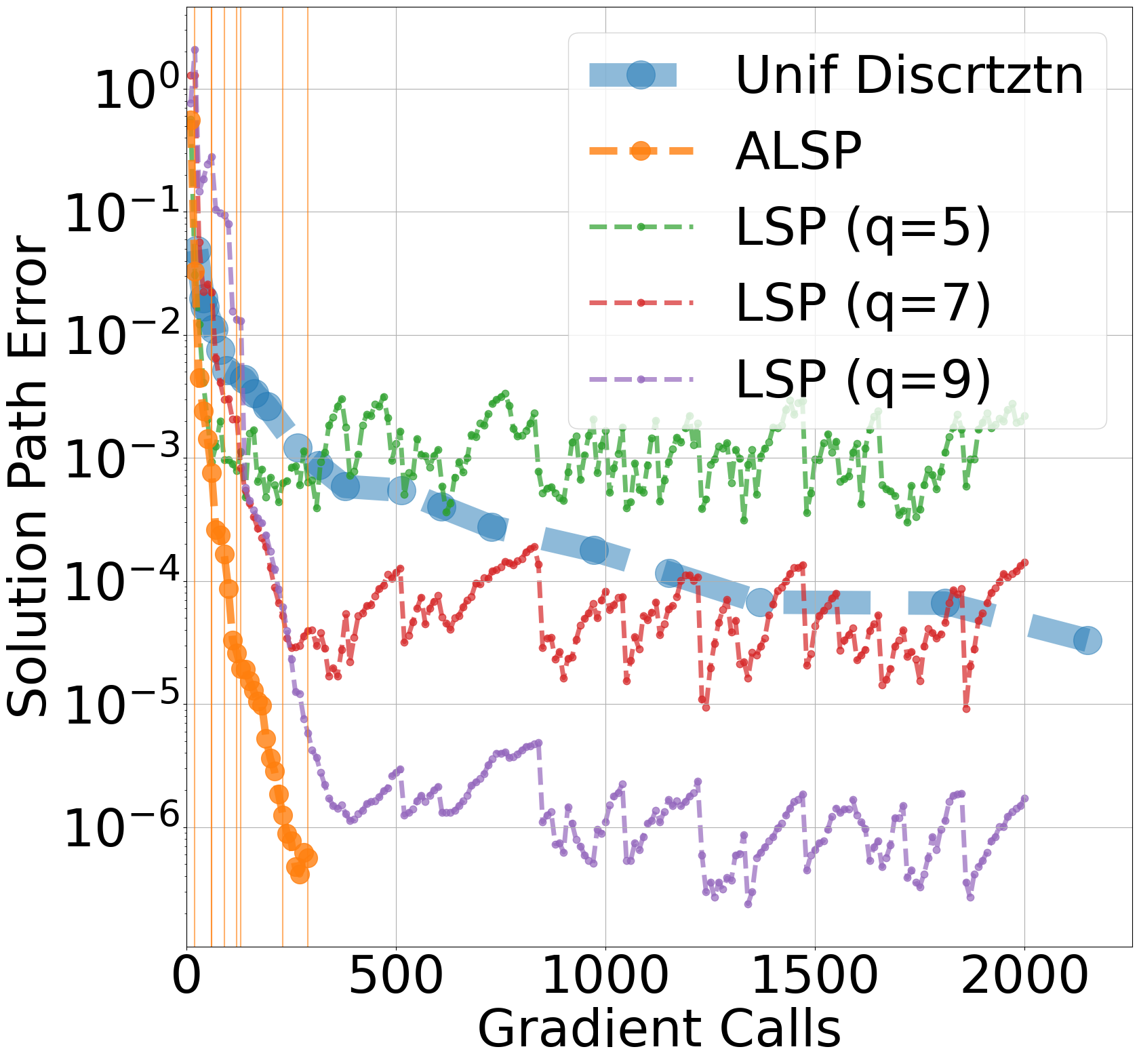}
\end{center}
\caption{\textbf{Learning Solution Path of Weighted Binary Classification.} See \cref{sec:experiment} for setup. We compare our method with $q = 5, 7, 9$ \emph{Legendre} polynomials as our basis with uniform discretization.  Orange line is our adaptive method (c.f. \cref{alg:adaptiveSgd}). 
 Vertical lines indicate when new basis functions are added.
\label{fig:preview}}
\end{figure}

Despite its simplicity, our approach can approximate the solution path to higher accuracy 
than discretization with fewer gradient evaluations.  See \cref{fig:preview} for a sample of our numerical results on weighted binary classification using SGD to solve \cref{eq:StochasticProb}.  
We prove this behavior is typical.   Loosely, 
\begin{enumerate}[label= \roman*)]
    \item \vskip -10pt When using constant step-size SGD to solve \cref{eq:StochasticProb}, we prove that 
the uniform error of our learned path $\bPhi_{1:p}(\lambda)\hat{\bbeta}$ to the true path $\btheta^*(\lambda)$ exhibits linear convergence to an irreducible constant that is proportional to the expressiveness of the basis (\cref{thm:exact1}).  This behavior is already visible in \cref{fig:preview}.
\end{enumerate}
\vskip -10pt
The proof of this result utilizes ideas from the convergence of SGD under various ``growth conditions" of the gradient \citep{bottou2018optimization, nguyen2018sgd, vaswani2019fast, liu2024aiming,bertsekas1996neuro}. See \citet{khaled2020better} for a summary and comparison of these various conditions. Our contribution to this literature is to relate the expressiveness of the basis $\bPhi_{1:p}(\lambda)$ to a relaxed, weak-growth condition of \citet[Lemma 2.4]{gower2019sgd}.  Indeed, we show that under some assumptions, \cref{eq:StochasticProb} always satisfies this relaxed weak-growth condition, and if the solution path lies in the span of the basis, it satisfies the weak-growth condition of \cite{vaswani2019fast}.  This allows us to leverage results from those works to establish \cref{thm:exact1}.

% \citet{vaswani2019fast} \qrd{\citet{khaled2020better} gives a nice summary of the growth condition concept along with comparison between ``strong/weak'' and ``relax''. Relaxed growth condition first appeared in Propsition 4.2 of \citep{}. A similar notion to our RWGC was used in \citet{}  as a direct consequence of expected smoothness.}:

% The proof requires generalizing results on the convergence of SGD in an interpolable regime (\cite{vaswani2019fast}) to a ``nearly" interpolable setting (\cref{sec:rwgc}).

In special cases, we can leverage a priori knowledge of the structure of $\btheta^*(\lambda)$ to prove stronger results.  For example, suppose $\Lambda = [-1,1]$, and we use Legendre polynomials as our basis. We prove that
\begin{enumerate}[resume,label= \roman*)]
    \item \vskip -10pt If the solution path $\btheta^*(\lambda)$ is $\nu$-times differentiable, then using $p = O\left(\epsilon^{\frac{1}{2(1-\nu)}}\right)$ 
    %\qrd{$O(d(\frac{dLV^2}{\nu^2})^{1/2\nu}((1+2\bar\eta/\mu)/\epsilon)^{1/2\nu} + d\nu)$}
    polynomials ensures that
    after $T = O\left(
    \epsilon^{\frac{1}{1-\nu}}\log(1/\epsilon)
    \right)$ gradient calls, we obtain an $\epsilon$-approximation to the solution path (\cref{thm:SGDnuDifferentiable}).
 % \qrd{$T = O((L/\mu)((\frac{dLV^2}{\nu^2})^{1/2\nu}((1+2\bar\eta/\mu)/\epsilon)^{1/2\nu} + \nu)^2 \log((1+2\bar\eta/\mu)/\epsilon))$}  
\end{enumerate}
\vskip -10pt
For comparison, \citet{ndiaye2019safe} implies that when $h(\btheta, \lambda)$ is strongly-convex, discretization requires at least $O(\epsilon^{-1/2})$ points. Hence, even if the optimization problem at each grid point can be solved with $1$ gradient evaluation, our approach requires asymptotically fewer evaluations whenever $\nu > 3$.  If $\nu$ is large, the savings is substantive.

\begin{enumerate}[resume,label= \roman*)]
    \item \vskip -10pt If the solution path $\btheta^*(\lambda)$ is analytic ($\nu = \infty$), then using $p = O( \log(1/\epsilon))$ basis polynomials ensures that after $T=O(\log^2(1/\epsilon) \log\log(1/\epsilon))$ iterations, we obtain an $\epsilon$-approximation to the solution path (\cref{thm:SGDforAnalytic}).
    % \vg{@Qiran:What constants are hidden here?}  
    % \qrd{$p = O((1/\log \omega) d \log(((1+2\bar\eta/\mu)/\epsilon)(dL)(M/(\omega-1))^2)$}
    % \qrd{$T = (L/\mu) $}
\end{enumerate}
\vskip -10pt
This is almost \emph{exponentially} less work than discretization for unidimensional hyperparameter space. 

These specialized results strongly suggest that in the general setting our approach should be competitive as long as we use enough basis functions.  
%Of course, in general, convergence to an irreducible constant is somewhat ``unsatisfying."  
%Simple discretization can recover the $\btheta^*(\lambda)$ to any accuracy with enough grid points.  
This observation motivates a natural heuristic in which we adaptively add basis functions whenever the stochastic optimization routine (e.g., SGD) stalls (see \cref{alg:adaptiveSgd} and \cref{sec:Adaptive} for details).  
We illustrate this idea in \cref{fig:preview} (orange line) where we can see that by progressively adding functions we drive the uniform error to zero rapidly.

\begin{algorithm}[tb]
\caption{Adaptively Learn the Solution Path (ALSP)}
\begin{algorithmic}[1]\label{alg:adaptiveSgd}
\STATE \textbf{Initialize:} $\bPhi_{1:0}(\cdot) \gets [\ ]$ and $\hat{\bbeta}_0  \gets [\ ]$
\STATE \textbf{Return:} A sequence of coefficients: $\hat{\bbeta}_1, \hat{\bbeta}_2, \ldots$
\vskip 5pt
\STATE {\bf At iteration $p = 1, 2, \ldots$ do}
\STATE \quad Append a new basis function: \vspace{-5pt}
    \[
    \bPhi_{1:p}(\cdot) \gets \left[\bPhi_{1:{p-1}}(\cdot)\quad \bPhi_p(\cdot)\right]
    \]
\STATE \quad Run a first-order method, starting from 
\STATE \quad $[\hat{\bbeta}_{p-1} , 0] \in \bbR^{p}$, to obtain 
    % With $[\hat{\bbeta}_{p-1} , 0] \in \bbR^{p}$ as the initialization, use a \INDSTATE stochastic first-order method to obtain
\vspace{-5pt}
    \begin{equation*}
        \hat{\bbeta}_p 
        ~\approx~
        \arg\min_{\bbeta \in \bbR^p} \  \bbE_{\lambdarv\sim\mathbb{P}_\lambda}\left[h(\bPhi_{1:p}(\lambdarv) \bbeta, \lambdarv)\right].
    \end{equation*}  
\end{algorithmic}
\end{algorithm}

\subsection{Other Related Work}
\label{sec:OtherRelatedWork}

There is a rich literature on specialized methods for computing the solution path under \emph{regularization}, i.e., where $\lambda \in \mathbb R_+$ is the weight on a convex regularizer.  These methods generally employ a path-following algorithm.  See, for example, \citet{rosset2004following,friedman2010regularization,10.1111/j.1467-9868.2007.00607.x} for cases where $\btheta^*(\lambda)$ is smooth, and \citet{rosset2007piecewise} for when it is piecewise linear.  The LASSO, or $\ell_1$ regularized regression, has received special emphasis \citep{osborne2000new,efron2004least,10.1214/11-AOS878},  as has support vector machines \citep{hastie2004entire} and certain structured regularizers \citep{bao2019efficient}.  \citet{liu2023new} adopts a slightly different perspective than this existing literature, using ordinary differential equations to motivate a second-order method to compute the path. Overall, like our work, these works generally consider the total computational time to compute the path, not just the number of discretization points.

However, our work differs from this literature in two important ways.  First, our proposed method applies to general $\lambda \in\Lambda$ which may be multidimensional, whereas the previous literature largely focuses on the case of $\lambda \in \mathbb R_+$. This limitation of the previous literature is perhaps fundamental as \emph{path}-following algorithms do not generalize easily to multidimensional settings.  
%(We \emph{do} present specialized results for our method in the unidimensional case to facilitate comparison to this literature.) 
Second, we treat a general function $h(\cdot, \cdot)$, not just the case of regularization. Despite our generality, our paper performs a relatively precise computational analysis by bounding total computation time in terms of the number of gradient calls, i.e., we treat the gradient as a block-box oracle instead of a more complicated black-box consisting of minimizing the objective $h(\cdot, \cdot)$ at fixed parameter values.
Furthermore, our numerical experiments go beyond simple regularization examples that are common in the literature by considering the upweighing of a minority class in classification and weighting different components of a multi-objective optimization problem.  Thus, our method applies more broadly than the aforementioned works.

%!TEX root = 0_Main.tex

\section{MODEL SETUP AND PRELIMINARIES}
\label{sec:Model}
We denote the $\ell_2$-norm by $\|\cdot\|$ throughout.
We focus on the case of smooth, strongly convex functions: 
\begin{assumption}[{\bf Uniform Smoothness and Strong Convexity}]\label{assu:regularity}
There are constants $0 < \mu \leq L$ such that, for all $\lambda \in \Lambda$, $h(\cdot, \lambda)$ is $\mu$-strongly convex and $L$-smooth, i.e., for all $\btheta, \bar\btheta \in \bbR^d$,
\begin{equation*}
\frac{\mu}{2}\|\btheta - \bar\btheta\|^2 ~\leq~ h(\btheta) - h(\bar\btheta) - \nabla h(\bar\btheta)^\top(\btheta - \bar\btheta) ~\leq~ \frac{L}{2}\|\btheta - \bar\btheta\|^2.
\end{equation*}
\end{assumption}

For any candidate solution path $\btheta(\cdot)$, we define the \emph{solution path error} of $\btheta(\cdot)$ by 
\begin{equation*}
    \epsilonsp(\btheta(\cdot)) := \sup_{\lambda\in \Lambda}\left\{h(\btheta(\lambda), \lambda) - h(\btheta^*(\lambda), \lambda)\right\}.
\end{equation*}
An \emph{$\epsilon$-solution path} is a solution path $\btheta(\cdot)$ such that $\epsilonsp(\btheta(\cdot)) < \epsilon$.  Finally, given any vector of coefficients $\bbeta$ and basis functions $\bPhi_{1:p}(\cdot)$, we define 
$\epsilonsp(\bbeta) := \epsilonsp(\bPhi_{1:p}(\cdot) \bbeta)$
to be the solution path error of $\bPhi_{1:p}(\cdot)\bbeta$.

% Thus, if $\epsilonsp(\hat{\bbeta}) \leq \epsilon$ for some $\hat{\bbeta}$, then $\thetahat(\cdot)$, defined by $\thetahat(\lambda) := \bPhi_{1:p}(\lambda)\hat{\bbeta}$, is an $\epsilon$-solution path.  

% \begin{definition}\label{def:sp_error}
%     ({\bf Solution Path Error and $\epsilon$-Solution Path}).
% \begin{enumerate}
%     \item For any solution path $\btheta(\cdot): \Lambda \mapsto \bbR^d$, we define the {\em solution path error} of $\btheta(\cdot)$ as
%     \begin{equation*}
%     \epsilonsp(\btheta(\cdot)) := \sup_{\lambda\in \Lambda}\left\{h(\btheta(\lambda), \lambda) - h(\btheta^*(\lambda), \lambda)\right\}.
% \end{equation*}

%     \item An \emph{$\epsilon$-solution path} $\thetahat(\cdot): \Lambda \mapsto \bbR^d$ is a solution path such that $\epsilonsp(\thetahat(\cdot)) < \epsilon$.
% \end{enumerate}

As mentioned, our method depends on two user-chosen parameters: a distribution $\mathbb P_\lambda$ over $\lambda \in \Lambda$ and a series of basis functions $\bPhi_j:\Lambda \mapsto \mathbb R^d$, for $j=1, \ldots, p$.  We require some minor assumptions on these choices:
\begin{assumption}\label{assum:components}
    ({\bf Distribution, Basis Functions, and Linear Independence}).
\begin{enumerate}[label = \roman*), itemsep=0pt]
    \item \vskip -10pt It is easy to generate i.i.d. samples from $\mathbb P_\lambda$.
    \item It is easy to compute $\bPhi_j(\lambda) \in \bbR^d$ for any $\lambda \in \Lambda$, and $1\leq j \leq p$.
    
    \item There does not exist 
    $\bbeta \in \bbR^p$ with $\bbeta \neq 0$ such that $\bPhi_{1:p}(\lambdarv)\bbeta = \boldsymbol 0$ holds $\mathbb{P}_\lambda$-almost everywhere.
\end{enumerate}
\end{assumption}
\vskip -10pt
This last condition is a linear independence assumption.  If violated, one can remove a function without affecting the span of the basis on the support of $\mathbb P_{\lambda}$.

Given $\bPhi_{1:p}$, we define the minimal  solution path error for this basis by
\begin{equation*}
    \epsilonsp^\ast := \inf_{\bbeta \in \bbR^p}\epsilonsp(\bbeta).
\end{equation*}
We also define the following auxiliary constants:
\begin{align}\label{def:bigClittleC}
    C &~:=~ \sup_{\lambda\in \Lambda} \sigma_{\max}\left(\bPhi(\lambda)^\top \bPhi(\lambda)\right),
\\  \notag
    c &~:=~ \sigma_{\min}\left\{\bbE_{\lambdarv \sim \mathbb{P}_\lambda}\left[\bPhi(\lambdarv)^\top \bPhi(\lambdarv)\right]\right\}.
\end{align}

%%%%%%%%%%%%%%%%%%%
%%%%%%%%%%%%%%%%%%%
%%%%%%%%%%%%%%%%%%%
In words, $C$ is a uniform bound on the largest eigenvalue of the positive semidefinite matrix $\bPhi(\lambda)^\top \bPhi(\lambda)$, and $c$ is the smallest eigenvalue of the corresponding expected matrix.
% are bounds on the largest and smallest \emph{eigenvalues} of the feature covariance matrix and its expectation respectively. 
By construction, $0 \leq c \leq C$.
% $\left(\bPhi(\lambda)^\top \bPhi(\lambda)\right)$ and $\bbE_{\lambdarv}\left[\bPhi(\lambdarv)^\top \bPhi(\lambdarv)\right]$ are positive semidefinite. 
Under \cref{assum:components}, both constants are strictly positive. 

\begin{lemma}[{\bf Positive Spectral Values}]\label{lma:cPositive}
If \cref{assum:components} holds, then $0 < c \leq C$.
\end{lemma}

We can now state our first key result which relates the suboptimality of a feasible solution $\bbeta$ in \cref{eq:StochasticProb} to its solution path error.  Define 
\begin{equation*} 
    \betaavg ~\in ~\arg\min_{\bbeta \in \bbR^p} \  \bbE_{\lambdarv\sim\mathbb{P}_\lambda}\left[h(\bPhi(\lambdarv) \bbeta, \lambdarv)\right].
\end{equation*}

\begin{theorem}[{\bf Relating Suboptimality to Solution Path Error}]\label{thm:decompose}
Under \cref{assum:components,assu:regularity}, for any $\bbeta \in \bbR^p$, we have
\begin{equation*}
\epsilonsp(\bbeta) ~\leq~ 2CL\|\bbeta - \betaavg\|^2 + \left(\frac{8CL}{c\mu}\right)\epsilonsp^\ast.
\end{equation*}
\end{theorem}
This bound decomposes into two terms, one proportional to $\|\bbeta - \betaavg\|^2$, which represents the suboptimality of $\bbeta$
and the other proportional to $\epsilonsp^\ast$, which measures the maximal expressiveness of the basis.  By solving \cref{eq:StochasticProb} to greater accuracy, we can drive down the first term, but we will not affect the second term.  To reduce the second term, we must add basis functions to obtain a better quality approximation.

\Cref{thm:decompose} shows any algorithm for solving \cref{eq:StochasticProb} can be used to approximate the solution path.  In the next section, we argue that when $\epsilonsp^\ast$ is small, constant step-size SGD for \cref{eq:StochasticProb} exhibits linear convergence to a constant proportional to $\epsilonsp^*$, making it an ideal algorithm to study.

\section{SGD TO LEARN THE SOLUTION PATH}\label{sec:convGuaranteeSGD}

% In this section, we show that stochastic gradient descent (SGD) nearly meets the convergence requirements outlined in \cref{prop:adaptiveConv} part 2 when applied to our ``stochastic counterpart'' problem \eqref{eq:StochasticProb}. This offers valuable insight into the effectiveness of ALSP.

In this section, we apply constant step-size SGD to solve \cref{eq:StochasticProb}.  The key idea is 
to show that \cref{eq:StochasticProb} satisfies a certain growth condition and apply \citet[Theorem 3.1]{gower2019sgd}. One minor detail is that \citet{gower2019sgd} proves their result under a stronger ``expected smoothness" condition for the setting where the objective \cref{eq:StochasticProb} is a finite sum.  However, the result holds more generally under a weaker condition (Eq. (9) of their work).  Hence, for clarity, we first restate this condition and the more general result.

\label{sec:rwgc}
% \paul{This subsection has several notation inconsistencies with the rest of the paper.}
% Before presenting our main result, we first present a standalone result on the performance of SGD when the objective function is nearly interpolable.  This result is an extension of \cite{vaswani2019fast} and underlies our main result in Section 3 and 4.

% \cref{def:rwgc} below generalizes the weak growth condition of  \citet{vaswani2019fast} \qrd{\citet{khaled2020better} gives a nice summary of the growth condition concept along with comparison between ``strong/weak'' and ``relax''. Relaxed growth condition first appeared in Propsition 4.2 of \citep{bertsekas1996neuro}. A similar notion to our RWGC was used in \citet{gower2019sgd} lemma 2.4 as a direct consequence of expected smoothness.}:

The following definition is motivated by \citet[Lemma 2.4]{gower2019sgd}:
\begin{definition}[{\bf Relaxed Weak Growth Condition (RWGC)}]\label{def:rwgc}
    Consider a family of functions $g(\cdot, z):\bbR^d \to \bbR$, $z\sim \bbP_z$, and $G(\cdot):= \bbE_{z\sim \bbP_z}\left[g(\cdot, z)\right]$. Let $G^* := \min_{\bw \in \bbR^d} G(\bw)$. Then, $g$ and $G$ are said to satisfy the \emph{relaxed weak growth condition} with constants $\rho \geq 0$ and $\sigma \geq 0$, if for all $\bw \in \bbR^d$, 
    \begin{equation*}
        \bbE_{z\sim \bbP_z}\left[\|\nabla g(\bw, z)\|^2\right] ~\leq~ 2\rho (G(\bw) - G^*) + \sigma^2.
    \end{equation*}
\end{definition}
%\begin{remark}[{\bf RWGC Explained}]
% We shall note that this definition tells us how large the gradient lengths of the family get on average. It is called a "growth" condition because it measures how much the gradient values deviate from some expectation, hence the coefficient $\rho \geq 1$.
When $\sigma = 0$, \cref{def:rwgc} recovers the \emph{weak growth condition} of \cite{vaswani2019fast}. In this case, the variance of the stochastic gradients goes to zero as we approach an optimal solution. Hence, at optimality, not only is the expectation of the gradient zero, but it is zero almost surely in $z$.  For regression problems, this condition corresponds to interpolation.

\cref{def:rwgc} is implied by a standard second moment condition on the gradients (take $\rho = 0$).  However, the most interesting cases are when $\sigma^2$ is small relative to $\rho$ and $\rho > 0$.  (This will be the case for \cref{eq:StochasticProb}. See \cref{sec:exact}.)

We then have the following theorem.
\begin{theorem}[\cite{gower2019sgd}]\label{lma:convergeWGC}
Suppose that the RWGC holds for the family of functions $g(\cdot, z):\bbR^d \to \bbR$, and that $G(\cdot):= \bbE_{z\sim \bbP_z}\left[g(\cdot, z)\right]$ is $\mu_g$-strongly convex for some $\mu_g > 0$. Consider the SGD algorithm, initialized at $\bw_0 \in \bbR^d$, with iterations
\begin{equation*}
    \bw_{t+1} = \bw_t - \frac{\bar{\eta}}{\rho} \nabla g(\bw_t, z_t),
\end{equation*}
where $z_t \sim \bbP_z$ is a random sample and 
%$\eta > 0$ is the step-size. If $\eta = \frac{\bar{\eta}}{\rho}$ is parameterized by some 
$\bar{\eta} \in \left(0, \min\left\{1, \frac{\rho}{\mu_g}\right\}\right]$ parametrizes the step-size.  {Let $\bw^*:= \arg\min_{\bw \in \bbR^d} G(\bw)$.}
Then for all $t \geq 0$, we have
\begin{equation*}
    \bbE[\|\bw_{t} - \bw^*\|^2] ~\leq~ \left(1 - \frac{\bar{\eta}\mu_g}{\rho}\right)^t\|\bw_{0} - \bw^*\|^2 + \frac{\bar{\eta} \sigma^2}{\rho\mu_g}.
\end{equation*}
% If $\sigma^2 > 0$, then for $t \geq \left\lceil\frac{\rho}{\bar{\eta}\mu_g}\log\left(\frac{\rho\mu_g\|\bw_{0} - \bw^*\|^2}{\bar{\eta}\sigma^2}\right)\right\rceil$, it holds that
% \begin{equation*}
% \bbE[\|\bw_{t} - \bw^*\|^2] ~\leq~ \frac{2\bar{\eta} \sigma^2}{\rho\mu_g}.
% \end{equation*}
% If $\sigma^2 = 0$, then for all $t \geq \left\lceil\frac{\rho}{\bar{\eta}\mu_g}\log\left(\frac{\|\bw_{0} - \bw^*\|^2}{\epsilon}\right)\right\rceil$ and $\epsilon > 0$, it holds that
% \begin{equation*}
% \bbE[\|\bw_{t} - \bw^*\|^2] ~\leq~ \epsilon.
% \end{equation*}
\end{theorem}
% The above result implies that when RWGC is satisfied, \cref{thm:decompose} immediately yields a bound on the expected solution path error. 
% What is more, if the following assumption also holds true, then we may obtain linear convergence without expectation with a high probability, which translates to a high probability bound directly on the solution path error.
In other words, constant step-size SGD under the RWGC exhibits a ``fast'' linear convergence in expectation up to a constant that is directly proportional to $\sigma^2$, after which it stalls.  When $\sigma^2 =0$, we exactly recover the result of \cite{vaswani2019fast}. To keep the exposition self-contained, we provide a proof in the appendix.

% \begin{assumption}[Light Tail Condition]\label{assu:lightTail}
%     For some fixed $a > 0$, we have for any $\bw \in \bbR^d$ that
%     \[
%     \Eb{\exp\left(\frac{\|\nabla g(\bw, z) - \nabla G(\bw)\|^2}{a^2}\right)} \leq \exp (1)
%     \]
% \end{assumption}
% \qrd{need help making this lemma work in order to provide support for bullet 2 in \cref{prop:adaptiveConv}}
% \begin{lemma}[High Probability Convergence under Light Tail Condition] \label{lma:convergeLightTail}
%     Suppose that functions $g(\cdot, z)$ and $G(\cdot)$ satisfy the assumptions in \cref{lma:convergeWGC} and inaddition, satisfy \cref{assu:lightTail}. Then, the SGD algorithm described in \cref{lma:convergeWGC} gives the following with probability at least 
%     $\left(1- \exp \left(\frac{-b^2}{3}\right) \right)\prod_{k=0}^{t-1} \left(1-\frac{e}{(k+1)^b}\right)$:
% \begin{align*}
%     \|\bw_{t} - \bw^*\|^2 
% &\leq 
%     (1 - \eta\mu_g)^t\|\bw_{0} - \bw^*\|^2 + \frac{\eta^2\sigma^2 + 2\eta D_W ab }{1 - \eta\mu_g} + \eta^2 a^2 b \sum_{k = 0}^{t-1} (1 - \eta\mu_g)^k \log(k+1)
% \end{align*}
% \qrd{$b$ and $D_W$ are place holder for some constants here and will be modified later. See definition in the proof. The last term needs a bound too.}

% \cref{lma:convergeLightTail} demonstrates that the second bullet assumption in \cref{prop:adaptiveConv} can be satisfied.

% \end{lemma}

\subsection{Solving \cref{eq:StochasticProb} with SGD}\label{sec:exact}
We next prove that \cref{eq:StochasticProb} satisfies RWGC, and in particular, that the ``$\sigma^2$'' term depends on $\epsilonsp^\ast$, the minimal solution path error of the basis.  To our knowledge, we are the first to make this observation.

% We next apply \cref{lma:convergeWGC} to analyze constant step-size SGD applied to \cref{eq:StochasticProb}.  
% now move on to develop convergence guarantees of SGD applied to our ``stochastic counterpart'' problem \eqref{eq:StochasticProb}, by applying 
% Our results hold for any $\epsilonsp^\ast > 0$, and lead to strong convergence guarantees that improve upon grid search methods under an interpolable basis function class with $\epsilonsp^\ast = 0$.  
% Recall our modified objective \cref{eq:StochasticProb}:
% \begin{equation*}
%     \min_{\bbeta \in \bbR^p} \  \bbE_{\lambdarv\sim\mathbb{P}_\lambda}\left[h(\bPhi(\lambdarv) \bbeta, \lambdarv)\right].
% \end{equation*}
% We propose to solve it with stochastic gradient descent randomly sampling $\lambdarv$.
% In order to apply SGD, we assume that the gradients of $h(\btheta, \lambda)$ are directly computable.
% \begin{assumption}[{\bf Gradient Oracle}]  \label{asn:ExactGradientOracle} We assume access to an oracle that given any $\lambda \in \Lambda$ and $\btheta \in \bbR^d$ returns $\nabla_{\btheta} h(\btheta, \lambda)$.
% \end{assumption}

% To apply \cref{lma:convergeWGC}, we must first confirm our problem satisfies RWGC.  To that end, 
Define 
% define the family of functions $f(\cdot, \lambda)$ and the expectation function $F(\cdot)$ that SGD is directly applied to. For any $\bbeta \in \bbR^p$ and $\lambda \in \Lambda$, these functions are defined by
\begin{subequations}\label{eq:auxiliary}
\begin{align}
    f(\bbeta, \lambda) &:= h(\bPhi(\lambda)\bbeta, \lambda),
\\    \notag
    F(\bbeta) &:=
    \bbE_{\lambdarv \sim \bbP_\lambda}\left[f(\bbeta, \lambdarv)\right] 
= \bbE_{\lambdarv\sim \bbP_\lambda}\left[h(\bPhi(\lambdarv)\bbeta, \lambdarv)\right],
\end{align}
\end{subequations}

% The above identity leads to the following implementation of SGD.
% \begin{algorithm}[H]
% \caption{Pseudo-code: SGD for Problem \eqref{eq:StochasticProb}}\label{alg:sgdExact}
% \begin{algorithmic}
% \STATE Initialize the starting input of Problem \eqref{eq:StochasticProb}, $\bbeta_0\in \bbR^{d\times p}$, arbitrarily.
% \FOR{ t = 0, 1, ..., T}
%     \STATE Sample $\lambdarv_t \sim\mathbb{P}_\lambda$,
%     \STATE $\bbeta_{t+1} \leftarrow \bbeta_t - \eta\bPhi(\lambdarv_t)^\top\nabla_{\btheta} h(\bPhi(\lambdarv_t)\bbeta_t, \lambdarv_t)$.
% \ENDFOR
% \RETURN $\bbeta_T$
% \end{algorithmic}
% \end{algorithm}
% We shall see that this algorithm produces an $\epsilon$-solution path under $O(\log(1/\epsilon))$ iterations.
% We next demonstrate that the RWGC and strong convexity properties are satisfied under our assumptions, for any $\epsilonsp^\ast \geq 0$.
% We will do so by showing that our objective function satisfies the RWGC, and therefore, has the convergence guarantee when we run SGD. 

% Recall
% \[
%     \epsilonsp^\ast 
%     := 
%     \inf_{\bbeta\in\bbR^p} \epsilonsp(\bbeta)
% \]
% where
% \[
%     \epsilonsp(\bbeta)
%     =
%     \sup_{\lambda \in \Lambda}\left(h(\bPhi(\lambda) \bbeta, \lambda) - h(\btheta^*(\lambda), \lambda)\right).
% \]

\begin{lemma}[{\bf \cref{eq:StochasticProb} satisfies RWGC}]\label{lma:satisfyWGC}
    Suppose that \cref{assum:components,assu:regularity} hold and recall the constants defined in \cref{def:bigClittleC}. Then, the family of functions $f(\cdot, \lambda)$ and the function $F(\cdot)$, defined in \cref{eq:auxiliary}, satisfy the relaxed weak growth condition (RWGC) with constants $\rho = CL$ and $\sigma^2 = 2CL\epsilonsp^\ast$. Namely, for all $\bbeta \in \bbR^d$, 
    \begin{equation*}
        \bbE_{\lambdarv\sim\bbP_\lambda}\left[\|\nabla f(\bbeta, \lambdarv)\|^2\right] ~\leq~  2CL(F(\bbeta) - F^*) + 2CL\epsilonsp^\ast.
    \end{equation*}
    In addition, $F(\cdot)$ is $c\mu$-strongly convex.
\end{lemma}
% \vg{@Qiran:  It might help when you clean this up to state the theorem in terms of $c, C, L, \mu$ instead of $\bar \mu$.  This makes the dependence on the ``condition number" $\frac{C}{c}$ a little more clear.}

We  now combine 
\cref{thm:decompose,lma:convergeWGC,lma:satisfyWGC} to yield our main result: a bound on the expected solution path error for constant step-size SGD.  Recall constant step-size SGD in this setting yields the iteration 
\[
\bbeta_{t+1} \leftarrow \bbeta_t - \eta\nabla_{\bbeta}h(\bPhi(\lambdarv_t)\bbeta_t, \lambdarv_t),
\]
where
\[
\nabla_{\bbeta}h(\bPhi(\lambdarv_t)\bbeta_t, \lambdarv_t) = \bPhi(\lambdarv_t)^\top \nabla_{\btheta} h(\bPhi(\lambdarv_t)\bbeta_t, \lambdarv_t),
\]
and $\lambdarv_t \sim \bbP_\lambda$.
We assume that the gradient $\nabla_{\btheta} h(\btheta, \lambda)$ is easily computable for any $\btheta$ and $\lambda$.
% We now combine the results of \cref{thm:decompose} and \cref{lma:convergeWGC,lma:satisfyWGC} to yield a bound on the expected solution path error for SGD. In particular, if we take $\kappa := \frac{CL}{c\mu}$ to be the condition number of \cref{eq:StochasticProb}, then when $\epsilonsp^\ast >0$, \cref{thm:exact1} says that SGD achieves an expected solution path error of $O(\kappa\epsilonsp^\ast)$ after at most $\tilde{O}(\kappa)$ iterations. When $\epsilonsp^\ast = 0$, \cref{thm:exact1} says that SGD achieves an expected solution path error of $O(\epsilon)$ after at most $O(\kappa\log(1/\epsilon))$ iterations for any $\epsilon > 0$.

We then have:
\begin{theorem}[{\bf Expected Solution Path Error Convergence for SGD}]\label{thm:exact1}
    Under \cref{assum:components,assu:regularity}, consider applying SGD to \cref{eq:StochasticProb} with a constant step-size $\eta = \frac{\bar{\eta}}{CL}$ parameterized by $\bar{\eta} \in \left(0, 1\right]$. Then, after $T$ iterations, the expected solution path error is at most
    \begin{align*}
        \Eb{\epsilonsp(\bbeta_T)}
    ~\leq~& 
        2CL\|\bbeta_0 - \betaavg\|^2 \left(1 - \frac{\bar{\eta}c\mu}{CL}\right)^T \\
        &+ \left(\frac{4CL(\bar{\eta} + 2)}{c\mu}\right)\epsilonsp^\ast.
    \end{align*}
    In particular, when $\epsilonsp^\ast > 0$, then for $T \geq \left\lceil\frac{CL}{\bar{\eta}c\mu}\log\left(\frac{c\mu\|\bbeta_0 - \betaavg\|^2}{2\bar{\eta}\epsilonsp^\ast}\right)\right\rceil$,
    \begin{equation*}
    \Eb{\epsilonsp(\bbeta_T)} ~\leq~ \left(\frac{8CL(\bar{\eta} + 1)}{c\mu}\right)\epsilonsp^\ast.
    \end{equation*}
    When $\epsilonsp^\ast = 0$, then for any $\epsilon > 0$ and $T \geq \left\lceil\frac{CL}{\bar{\eta}c\mu}\log\left(\frac{2CL\|\bbeta_0 - \betaavg\|^2}{\epsilon}\right)\right\rceil$,
    \begin{equation*}
    \Eb{\epsilonsp(\bbeta_T)} ~\leq~ \epsilon.
    \end{equation*}
\end{theorem}

\Cref{thm:exact1} highlights the role of the basis functions in our results. First we see that as $T\rightarrow \infty$, the expected solution path error plateaus at a constant proportional to $\frac{CL}{c\mu} \epsilon^\ast_{\rm sp}$.  As mentioned,   $\epsilon^\ast_{\rm sp}$ measures the expressiveness of the basis, and we expect $\epsilon_{\rm sp}^*$ to be small as we add more basis functions. We interpret $\frac{CL}{c\mu}$ as a condition number for \cref{eq:StochasticProb}, which also depends on the choice of basis through \cref{def:bigClittleC}.  This constant increases as we grow the basis.  Finally, the iteration complexity scales linearly with this condition number.  
Thus, an ``ideal" basis must navigate this tradeoff between $\epsilon^\ast_{\rm sp}$ and $C/c$.

Fortunately, there exists a rich theory on function approximation that studies the relationship between basis functions, uniform  error, and eigenspectra.  In the next section we leverage this theory to provide a comparison of our method with existing discretization techniques in a specialized setting.

%!TEX root = 0_Main.tex

% \section{Preliminaries}

% \paul{Proposed Changes to Structure:  Move 3.2 to be part of the now Section 4 (and call this Section 3). Move 3.1 Later, maybe to a new Section 4? Idea is to focus on results for fixed basis functions first that show how optimization error can be controlled, and then have a "high level" discussion of the interplay between choice of basis functions and approximation error.}
% Auxiliary results useful to our paper.

\section{SPECIALIZED RESULTS FOR $\Lambda = [-1, 1]$}\label{sec:SpecializedResults}

We next leverage results from function approximation theory to bound the number of basis functions needed to achieve a target solution path error.  We focus on the case $\Lambda =[-1,1]$ as it facilitates simple comparisons to existing results and elucidates key intuition.  

We use a simple basis: we approximate each component $i =1, \ldots, d$ of $\theta^*_i(\cdot)$ by the first $q$ Legendre polynomials.  Hence the total number of basis functions is $p=qd$.  Recall, the Legendre polynomials form an orthogonal basis on $[-1, 1]$ with respect to the uniform distribution, i.e., $\mathbb E_{\lambdarv \sim \text{Unif}[-1, 1]}\left[P_n(\lambdarv) P_m(\lambdarv)\right] = \frac{2}{2n+1}\Ib{n = m}$, where $P_n$ and $P_m$ are the $n^\text{th}$ and $m^\text{th}$ Legendre polynomial, respectively.  Since we approximate each component separately, the matrix $\bPhi(\lambda) \in \mathbb R^{d \times qd}$ is block-diagonal, with d blocks of size $1 \times q$.

For this basis, depending on the value of $p = qd$ implied by the choice of $q$, let $C_p, c_p$ refer to the constants \eqref{def:bigClittleC}. We can bound the constant $C_p/c_p$:
\begin{lemma}[\textbf{$C_p/c_p$ for Legendre Polynomials}] \label{lem:CpcpLegendre}
Take $\Lambda = [-1, 1]$ and $\mathbb P_{\lambda}$ to be the uniform distribution on $[-1, 1]$.  Then, for the above basis, $C_p/c_p \leq q^2.$ 
\end{lemma}
Notice, that this constant is independent of $d$ and grows mildly with $q$.  This is not true of all polynomial bases.  One can check empirically that for the monomial basis, $C_p/c_p$ grows exponentially fast in $q$.

Bounding $\epsilon^*_{\rm sp}$ depends on the properties $\btheta^*(\lambda)$. As a first example,

\begin{lemma}[\textbf{$\epsilon^*_{\rm sp}$ for $\nu$-Differentiable Solution Paths}]\label{lma:chebyshev}
Let $\Lambda = [-1, 1]$. Suppose  \cref{assu:regularity} holds. Further assume that there exists an integer $\nu \geq 0$ and constant $V > 0$ such that for all $i=1, \ldots, d$, $\btheta^*_i(\cdot)$ has $\nu$ derivatives,
% and the $\nu^\text{th}$ derivative $\btheta_i^{\ast (\nu)}$ has total variation bounded by $V$. 
{where $\btheta^{*}_i(\lambda)$, ..., $\btheta^{(\nu-1)*}_i(\lambda)$ are absolutely continuous and $\btheta^{(\nu)*}_i(\lambda)$ has total variation bounded by $V$.}
Then, for any $q \geq \nu+1$, for the basis described above, 
\(
\epsilon^*_{\rm sp} \ \leq \ 
    \frac{dL}{2}\left(\frac{2V}{\pi \nu(q-\nu)^\nu}\right)^2.
\)
\end{lemma}
The proof of \cref{lma:chebyshev} is constructive; we exhibit a polynomial with the given solution path error.  
{The bound confirms the intuition that if $\btheta^*(\cdot)$ is smooth (has many derivatives), then adding basis function drives down the approximation error rapidly.  
Problems arising in many application domains, including machine learning problems like ridge regression and relatives, often exhibit highly or even infinitely differentiable solution paths.}

Combining these lemmas with \cref{thm:exact1}
allows us to calculate the requisite basis size and number of iterations needed to 
achieve a target solution-path error:
\begin{theorem}[\textbf{SGD for $\nu$-differentiable Solution Paths}]
\label{thm:SGDnuDifferentiable}
 Let $\Lambda = [-1, 1]$, $\mathbb P_\lambda$ be the uniform distribution on $[-1, 1]$.  Then, under the conditions of \cref{lma:chebyshev} and assume $\nu \geq 2$, if we use $q = O(\epsilon^{\frac{1}{2(1-\nu)}})$ polynomials in the previous basis, and run constant-step size SGD for $O(\epsilon^{\frac{1}{1-\nu}} \log(1/\epsilon))$ iterations, the resulting iterate $\bbeta_T$ satisfies $\Eb{\epsilon_{\rm sp}(\bbeta_T)} \leq \epsilon$.  
\end{theorem}
{For clarity, while \cref{lma:chebyshev} holds for any $\nu \geq 0$, \cref{thm:SGDnuDifferentiable} requires $\nu \geq 2$.  Moreover,} both big ``Oh'' terms should be interpreted as $\epsilon \rightarrow 0$, and both suppress constants not depending on $\epsilon$ (but possibly depending on $h$).  The theorem establishes that the ``smoother" $\btheta^*(\cdot)$ is (i.e. larger $\nu$), the fewer iterations required by our method to achieve a target tolerance.

Recall, \cite{ndiaye2019safe} established that for strongly convex functions, discretization requires at least $O(\epsilon^{-1/2})$ points.  One might expect the number of gradient evaluations per point to scale like $O(\log(1/\epsilon))$.  Hence, if $\nu \geq 3$, our approach requires asymptotically less work.  The larger $\nu$, the larger the savings.

This gap becomes more striking as $\nu \rightarrow \infty$.  Hence, we next consider the case where $\theta^*_i(\lambda)$ is analytic on $[-1, 1]$, i.e, its Taylor Series is absolutely convergent on this interval.  
\begin{lemma}[\textbf{$\epsilon^*_{\rm sp}$ for Analytic Solution Paths}] \label{lma:analytic}
Suppose \cref{assu:regularity} holds and that for each $i=1, \ldots, d$, $\btheta^*_i(\cdot)$ is analytic on the interval $[-1, 1]$. Then, there exist constants $\omega > 1$ and $M > 0$ such that for any $q > 0$, with the basis described above, 
\(
\epsilon^*_{\rm sp} \leq \frac{dL}{2}\left(\frac{2M \omega^{-q}}{\omega - 1}\right)^2.
\)
\end{lemma}
The proof is again constructive. The constants $\omega$ and $M$ pertain to the analytic continuation of $\theta^*_i(\cdot)$ to the complex plane.  Importantly, the solution path error now dies geometrically fast (like $\omega^{-2q}$).  Using this faster decay rate yields,
\begin{theorem}[\textbf{SGD for Analytic Solution Paths}] \label{thm:SGDforAnalytic}
Let $\Lambda = [-1,1]$, and $\mathbb P_{\lambda}$ be the uniform distribution on $[-1,1]$. Then, under the conditions of \cref{lma:analytic}, if we use $q = \log(1/\epsilon)/\log(\omega)$ polynomials 
in the previous basis, and run constant-step size SGD for $O\left(\log^2(1/\epsilon) \log\log(1/\epsilon)\right)$ iterations, the resulting iterate $\bbeta_T$ satisfies $\Eb{\epsilonsp(\bbeta_T)} \leq \epsilon$.  
\end{theorem}
Again, the big ``Oh" hides constants that do not depend on $\epsilon$. Compared to \cite{ndiaye2019safe}, the amount of work required is almost exponentially smaller for $\dim(\Lambda) = 1$. 

% In higher dimensions, the curse of dimensionality implies that the number of discretization points required for accuracy grows exponentially with $\dim(\Lambda)$. In contrast, our method's performance does not explicitly depend on $\dim(\Lambda)$ and depends instead on the minimal solution path error $\epsilon_\mathrm{sp}^*$. The key advantage arises when we either select a basis ensuring a small $\epsilon_\mathrm{sp}^*$ using prior knowledge or adopt a highly flexible, adaptive approach, as we next discuss in \cref{sec:Adaptive}.

In higher dimensions, the curse of dimensionality causes the required discretization points to grow exponentially with \(\dim(\Lambda)\). In contrast, our method’s performance depends not on \(\dim(\Lambda)\) but on the minimal solution path error \(\epsilon_\mathrm{sp}^*\). The key advantage arises when we either select a basis ensuring a small $\epsilon_\mathrm{sp}^*$ using prior knowledge or adopt a highly flexible, adaptive approach, as we next discuss in \cref{sec:Adaptive}.

{

\section{IMPLEMENTATION GUIDELINES}\label{sec:Adaptive}
Although \cref{sec:SpecializedResults} provides insight on how to select a basis (in certain cases), choosing the ``right'' basis apriori from approximation theory remains a difficult challenge. To circumvent this issue, we re-introduce \cref{alg:adaptiveSgd}. \cref{alg:adaptiveSgd} is more practical, as we only need to specify: i) a sequence of basis functions, ii) a distribution $\mathbb P_\lambda$ and, iii) a criterion for deciding when to add a new function to the basis. We next provide intuition and practical guidance on these choices.  
% \vskip 6pt

{\bf Intuition for an ``Ideal" Basis.}
Consider
\[
h(\btheta, \blambda)\ =\ \btheta^\top \bm Q(\blambda) \btheta + \bm b(\blambda)^\top \btheta.
\]
In this special case, \cref{eq:StochasticProb} becomes
\begin{equation*}
    \min_{\bm \beta} \bm \beta^\top \Eb{\bPhi(\blambdarv)^\top \bm Q(\blambdarv) \bPhi(\blambdarv) } \bbeta 
    + \Eb{\bm b(\blambdarv)^\top \bPhi(\blambdarv)} \bbeta.
\end{equation*}
The complexity of this problem depends strongly on its condition number (in our theory, this is bounded by $\frac{CL}{c\mu}$), which equals the condition number of the matrix 
\(
\Eb{\bPhi(\blambdarv)^\top \bm Q(\blambdarv) \bPhi(\blambdarv) }
\)
and is directly computable in this case.
Hence, an ideal basis would ensure this condition number is as close to $1$ (its lower bound) as possible. In particular, if the basis functions $\bPhi_j(\cdot)$ are orthonormal with respect to the inner product
\[
\langle \bPhi_j(\cdot), \bPhi_k(\cdot) \rangle \ \equiv \ \bbE_{\lambdarv\sim \bbP_\lambda}\left[\bPhi_j(\blambdarv)^\top \bm Q(\blambdarv) \bPhi_k(\blambdarv)\right],
\]
then, the resulting condition number is $1$.  
For general $h$, insofar as $F(\bbeta) = \bbE_{\lambdarv\sim \bbP_\lambda}\left[h(\bPhi(\lambdarv)\bbeta, \lambdarv)\right]$ might be approximated by a second-order Taylor series expansion around the optimum $\betaavg$, a good basis might be orthonormal w.r.t. to the Hessian of $F(\cdot)$ at optimality. We verify this in \cref{sec:highDim}. 
Absent knowledge about the Hessian at optimality, the best approach is to use a family of flexible and expressive basis functions.  
In \cref{sec:weighted_logit,sec:portfolioAlloc}, we use orthogonal polynomials for the following reasons: i) Adding new basis functions does not require altering the existing basis, and ii) Intuitively, orthogonality suggests that if $\hat\bbeta_{p-1}$ is near optimal in the $(p-1)^\text{th}$ iteration, then $[\hat \bbeta_{p-1}, 0]$ is likely to be a near optimal (and good warm-start) solution in the $p^\text{th}$ iteration. The first benefit is not shared, for example, by other approaches like cubic splines or neural networks.
% This is unlike, e.g., cubic splines or adding nodes to a neural network, where the addition of a new spline point affects previous bases.
% \begin{enumerate}[label= \roman*)]
%     \item \vskip -10pt Extending the basis does not require altering existing basis.  This is unlike, e.g., cubic splines or adding nodes to a neural network, where the addition of a new spline point affects previous bases.
%     \item \vskip -10pt Intuitively, orthogonality suggests that if $\hat\bbeta_{p-1}$ is near optimal in the $(p-1)^\text{th}$ iteration, then $[\hat \bbeta_{p-1}, 0]$ is likely to be close to optimal in the $p^\text{th}$ iteration.  Thus, we benefit from warmstarts.
% \end{enumerate}
% \vskip -10pt

{\bf Interplay Between Basis and Distribution.}
Although \emph{any} sequence of basis functions and \emph{any} distribution can be used in \cref{alg:adaptiveSgd}, we suggest making these choices in concert.  In our experiments, we focus on sequences of polynomials that are orthogonal with respect to $\mathbb  P_{\lambda}$. The Legendre polynomials and the uniform distribution on $[-1, 1]$ is one such pair, but there are many canonical examples including Hermite polynomials with the normal distribution and Laguerre polynomials with the exponential distribution. There are performance-optimized implementations of these families in standard software (see, e.g., \texttt{scipy.special}). 
Although polynomials map to $\mathbb R$, 
by approximating each dimension separately as in \cref{sec:SpecializedResults}, we can extend polynomials to a basis for $d > 1$.  Finally, polynomials are highly expressive as they are uniform approximators for Lipschitz functions.

% \vskip 8pt

{\bf Deciding the Number of Basis Functions ``On the Fly".}
As reflected by our theoretical convergence guarantees, choosing a large, flexible basis induces a tradeoff.  A larger basis increases the condition number of \cref{eq:StochasticProb} (as measured by the ratio $\frac{CL}{c\mu}$), and hence convergence will be slower. This observation motivates our adaptive approach that only adds basis functions as necessary until we reach a desired accuracy.  One approach is to empirically approximate the objective of \cref{eq:StochasticProb} using a hold-out validation set, and add a basis function when performance stalls. 
% This approach requires additional functional evaluations and samples $\lambdarv$.  In our experiments, we explore an alternate criterion motivated by the RWGC.  Specifically, if $\bbeta$ is near-optimal, the  upperbound on the second moment of the gradient in RWGC is approximately constant.  Thus, we add a basis function when second moment of the gradient averaged over recent iterations has plateaued.  This allows us to reuse gradients and additional function evaluations.
In our experiment with low dimensional hyperparameters (\cref{sec:weighted_logit} and 6.2), we use Gauss-Legendre quadrature to evaluate the stochastic objective exactly over the entire hyperparameter space. In the high dimensional hyperparameter experiment (\cref{sec:highDim}), we instead compute the performance of our learned $\bbeta$ through an ERM version of the stochastic objective on a validation set of size $1000$.

}
% 
% \input{Noisy Gradients}

% \input{Practical Implementation}

%!TEX root = 0_Main.tex
\section{NUMERICAL EXPERIMENTS}\label{sec:experiment}

% {\color{red}
% Experiment for computing the Pareto front:
% \begin{itemize}
%     \item \texttt{https://citeseerx.ist.psu.edu/document?repid=rep1&type=pdf&doi=f329eb18a4549daa83fae28043d19b83fe8356fa}
    
%     Section 3.3. $\Lambda$ is 2-dimensional. Pro: easier/faster to code; con: too simple?
    
%     \item \texttt{https://arxiv.org/pdf/2008.10797.pdf}
    
%     Page 11 section 5.2. Real data with real purpose (accuracy vs other fairness metrics). $\Lambda$ is more than 2-d. con: takes long to code
% \end{itemize}
% }

% In this section, we empirically validate our theoretical results. 

We next empirically compare 
i) our approach using a fixed, large basis (denoted \emph{LSP} for ``Learning the Solution Path") 
ii) our adaptive \cref{alg:adaptiveSgd} (denoted \emph{ALSP} for ``Adaptive LSP" and 
iii) uniform discretization, a natural benchmark. 
We aim to show that both our approaches not only outperform the benchmark but that the qualitative insights from our theoretical results hold for more general optimization procedures than constant step-size SGD.  We also provide preliminary results on how the choice of basis affects performance.
Our repository can be found at \url{https://github.com/Cumberkid/Learning-the-Optimal-Solution-Path}. 

We choose uniform discretization to be our benchmark in lieu of other schemes (like geometric spacing) because i) it matches the theoretical lower bound from \citep{ndiaye2019safe} and ii) we see it as most intuitive for learning the solution path with small \emph{uniform} error.  Specifically, for various $\epsilon$, we consider a uniform spacing of size $\sqrt{\epsilon}$ in each dimension of $\lambda$, and run (warm-started) gradient descent for $O(\log(1/\epsilon))$ iterations.  The constant hidden by big ``Oh'' here is calibrated in an oracle fashion to achieve a solution-path error of $O(\epsilon)$ (see \cref{sec:fixScheduleNGS} for details) giving the benchmark a small advantage.

\subsection{Weighted Binary Classification}\label{sec:weighted_logit}
We consider a binary classification problem using a randomly selected subset of $1000$ cases from
the highly imbalanced Law School Admission Bar Passage dataset \citep{wightman1998lsac}. 
Of the $1000$ cases, there are $956$ positive instances and $44$ negatives. Standard logistic regression predicts $992$ positives with a false positive rate of $0.86$. When identifying students likely to fail is key, the default classifier may not be useful.  Reweighting cases is a standard approach to improve false positive rate at the cost of overall accuracy.
% If our aim in predicting the Bar exam outcome for students is to provide additional support to those predicted to fail, then this default classifier proves not very useful. Overweighting the failure cases presents a viable strategy for training a more useful classifier, albeit the appropriate weight remains unclear. This scenario underscores the importance of obtaining the entire solution path across all potential weights on the positive/negative classes.

We take 
% \[
%     h(\btheta, \lambda) = (1-w(\lambda)) l_\text{pos}(\btheta) + w(\lambda) l_\text{neg}(\btheta) + 0.125 \|\btheta\|^2,
% \]
\[
    h(\btheta, \lambda) = (1-\lambda) l_\text{pos}(\btheta) + \lambda l_\text{neg}(\btheta) + 0.125 \|\btheta\|^2,
\]
where $l_\text{pos}(\btheta)$ and $l_\text{neg}(\btheta)$ denote the negative log-likelihood on the positive and negative classes respectively.  Specifically, letting $(\bx_i, y_i) \in \mathbb R^{45} \times \{0, 1\}$ for $i=1, \ldots, n$ denote the data, 
    \[
        l_\text{pos}(\btheta)
        = \frac{1}{|\{i:\ y_i = 1\}|} \sum_{i: y_i = 1} \log(1+e^{(-2y_i+1)\bx_i^\top\btheta}),
    \]
and similarly for $l_{\text{neg}}(\btheta)$. 

% The hyperparameter $\lambda \in [0,1]$ controls the weight placed on each class. 

% We use (scaled and shifted) Legendre polynomials with a $\text{Unif}[0,1]$ distribution.

% We consider a parametric logistic regression formulation that places different weights on the logistic loss for pass (positive) and failure (negative) cases.
%     % Define $l(\btheta)$ as the log-likelihood of binary logistic regression at $\btheta$ 
%     With feature data $\{\bx_i\}_{i = 1}^n$ and corresponding class labels $\{y_i\}_{i = 1}^n$ with $y_i \in \{0, 1\}$, the problem objective function is:
%     \[
%     h(\btheta, \lambda) = w_\text{pos} l_\text{pos}(\btheta) + w_\text{neg} l_\text{neg}(\btheta) + w_\text{const} \|\btheta\|^2,
%     \]
%     with $w_\text{const}$ being an absolute constant, and :
%     % \[
%     %     l(\btheta)
%     %     = \frac{1}{n} \sum_{i=1}^n y_i \log(1+e^{-\bx_i\btheta}) + (1-y_i) \log(1 + e^{\bx_i\btheta}).
%     % \]
%     \[
%         l_\text{pos}(\btheta)
%         = \frac{1}{|\{i:\ y_i = 1\}|} \sum_{i: y_i = 1} \log(1+e^{(-2y_i+1)\bx_i^\top\btheta}),
%     \]
%     and likewise for $l_\text{neg}(\btheta)$.

We consider two different choices of distribution for $\lambda\in[0, 1]$ together with two different orthogonal polynomial bases:
\begin{enumerate}[label= \roman*)]
    \item \vskip -10pt A $\text{Unif}[0,1]$ distribution.  Here we use (scaled and shifted) Legendre polynomials.
    \item \vskip -10pt A $\text{Beta}(b+1, a+1)$ distribution.  Here we use (shifted) Jacobi polynomials with parameters $(a, b) = (-.3, -.7)$.
    % (cf. \cref{sec:jacobi}.)
\end{enumerate}
\vskip -10pt
% These represent two parameterizations for (essentially) the same family of optimization problems.  Note that in the Laguerre case, a uniform discretization in $\lambda$ induces a non-uniform discretization in $w(\lambda)$ with large spacing near $0$.  

%Although we do not pursue this here, the interplay between the parametrization of the optimization family, choice of basis, and choice of distribution is an interesting avenue of future research.

The ground truth $\btheta^*(\lambda)$  is computed via $5000$ iterations of (warm-started) gradient descent over a uniform grid of $2^{10}$ points. Solution path error is approximated by the uniform error over this grid.

For LSP, we run SGD using \texttt{torch.optim.SGD}.  In lieu of a constant learning rate, we dynamically reduce the learning rate according to the 
Distance Diagnostic of \citet[Section 4]{pesme2020convergence}. This dynamic updating is more reflective of practice.  We use the suggested parameters from \citet{pesme2020convergence} with the exception of $\mathtt{q}$\footnote{Note that $\mathtt{q}$, defined as a diagnostic frequency control parameter in \citet[Section 4]{pesme2020convergence}, is different from $q$, the number of polynomials in our paper.}, which we tune by examining the performance after $200$ iterations.  (We take $\mathtt{q}=1.3$ for both the Legendre and Jacobi bases.)  
For ALSP, we initialize the algorithm with $5$ polynomials and stop it after reaching $12$ polynomials ($q=5$ to $12$).

% For LSP, we record the approximated solution path error
% as a function of the number of gradient calls.
% For uniform discretization, at each value of $\delta$, we count the \emph{total} number of gradient calls across all $1/\sqrt{\delta}$ grid points, and record the approximate solution path error incurred by the interpolated solution path along the chosen grid.

\begin{figure}[hbt!]
     \centering
          \begin{subfigure}[b]{0.225\textwidth}
         \centering
         \includegraphics[width=\textwidth]{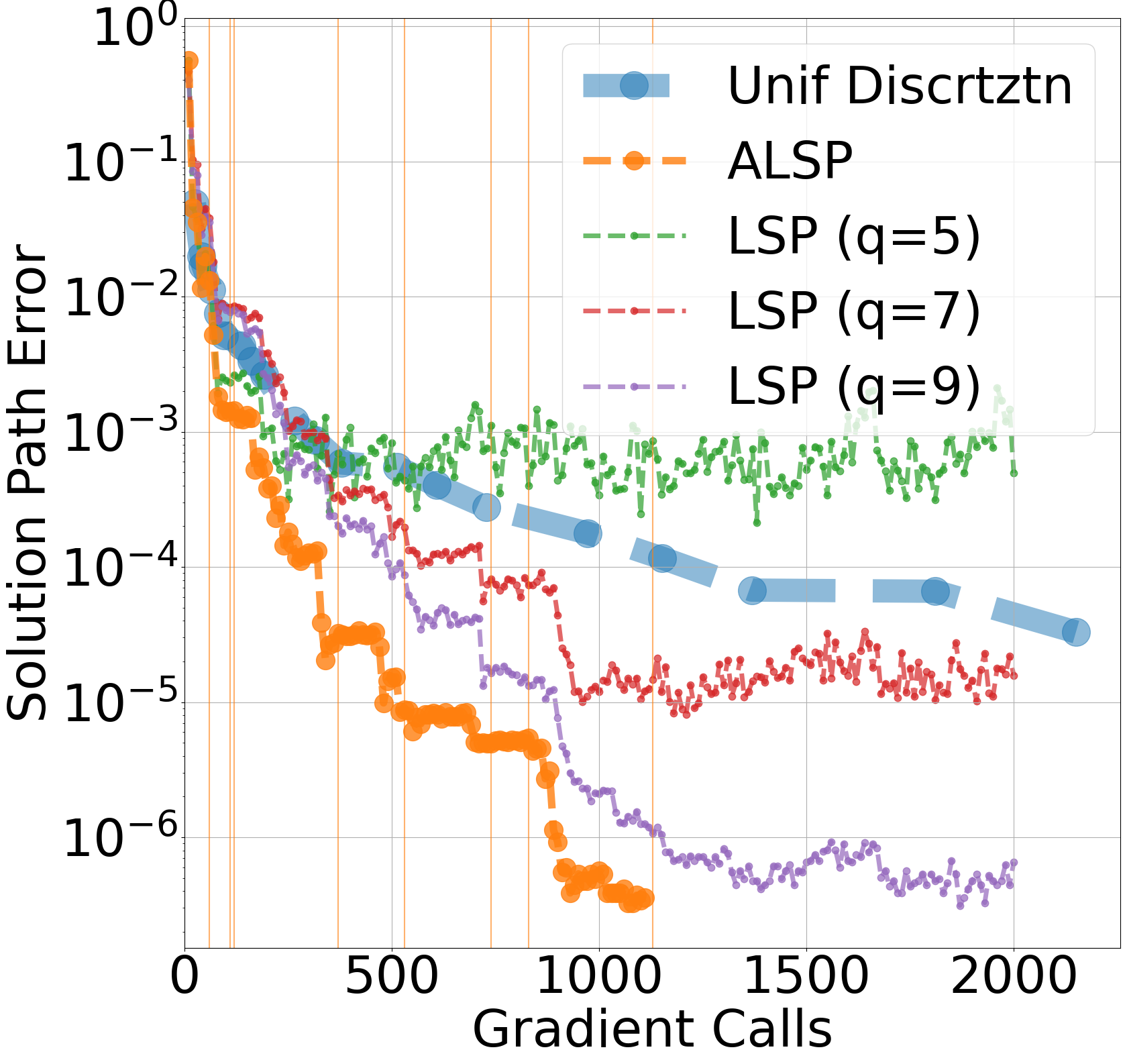}
         \caption{}
         \label{fig:boosted_and_fixed_exact_jacobi}
     \end{subfigure}
     \hfill
    \begin{subfigure}[b]{0.245\textwidth}
         \centering
         \includegraphics[width=\textwidth]{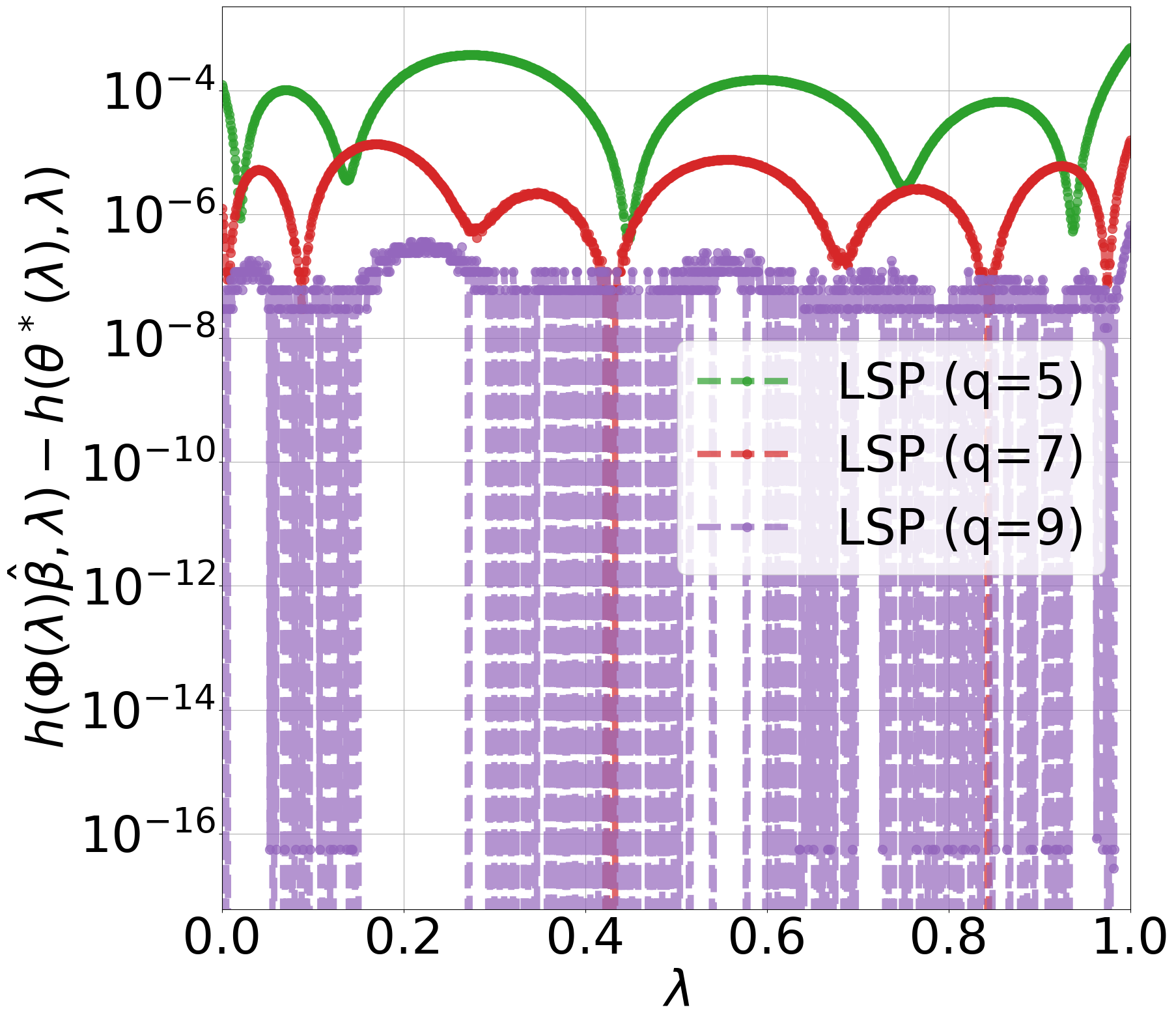}
         \caption{}
         \label{fig:err_along_soln_path_jacobi}
     \end{subfigure}
     \hfill
     \begin{subfigure}[b]{0.2375\textwidth}
         \centering
         \includegraphics[width=\textwidth]{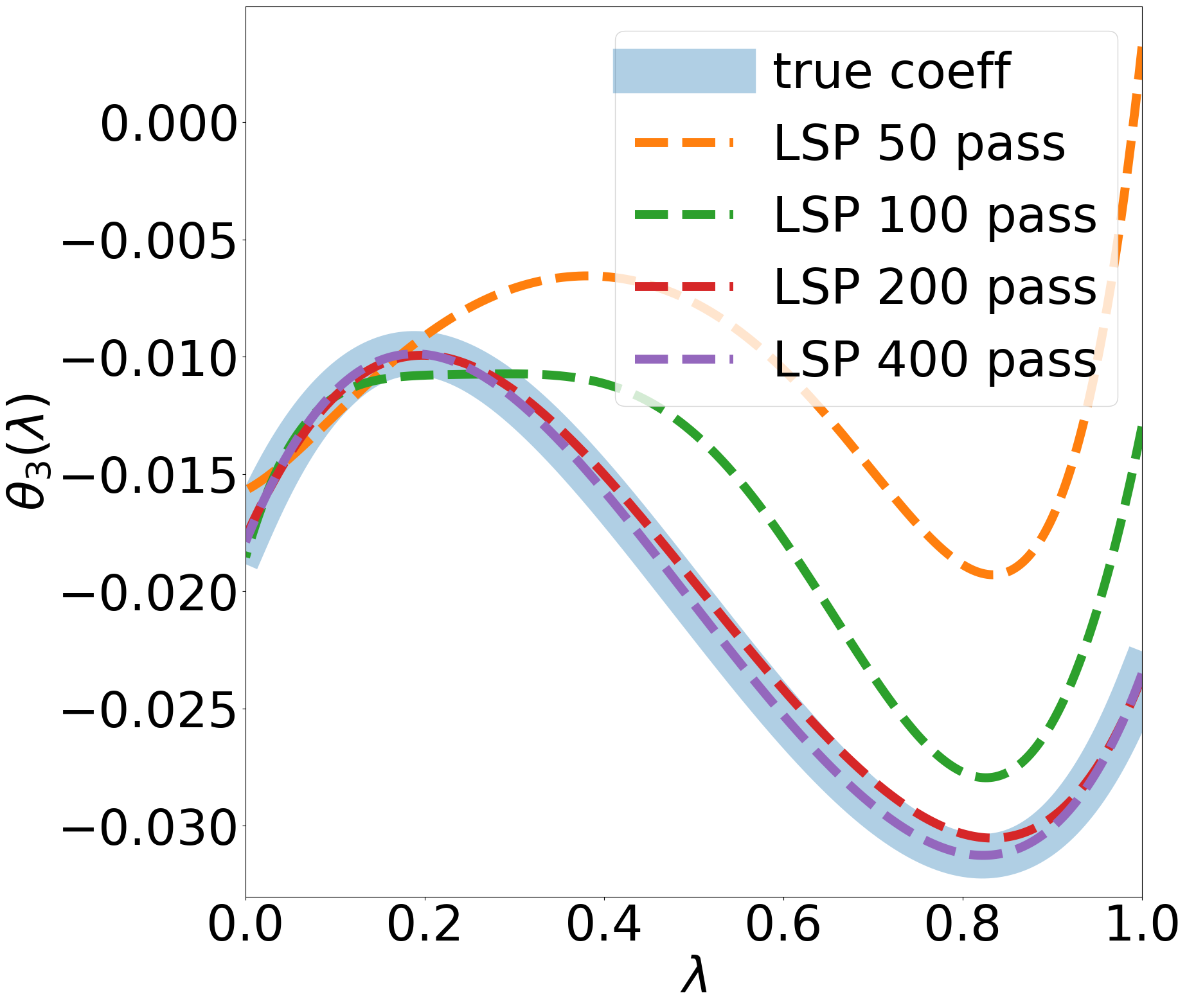}
         \caption{}
         \label{fig:CPP_exact_legendre}
     \end{subfigure}
     \hfill
     \begin{subfigure}[b]{0.2375\textwidth}
         \centering
         \includegraphics[width=\textwidth]{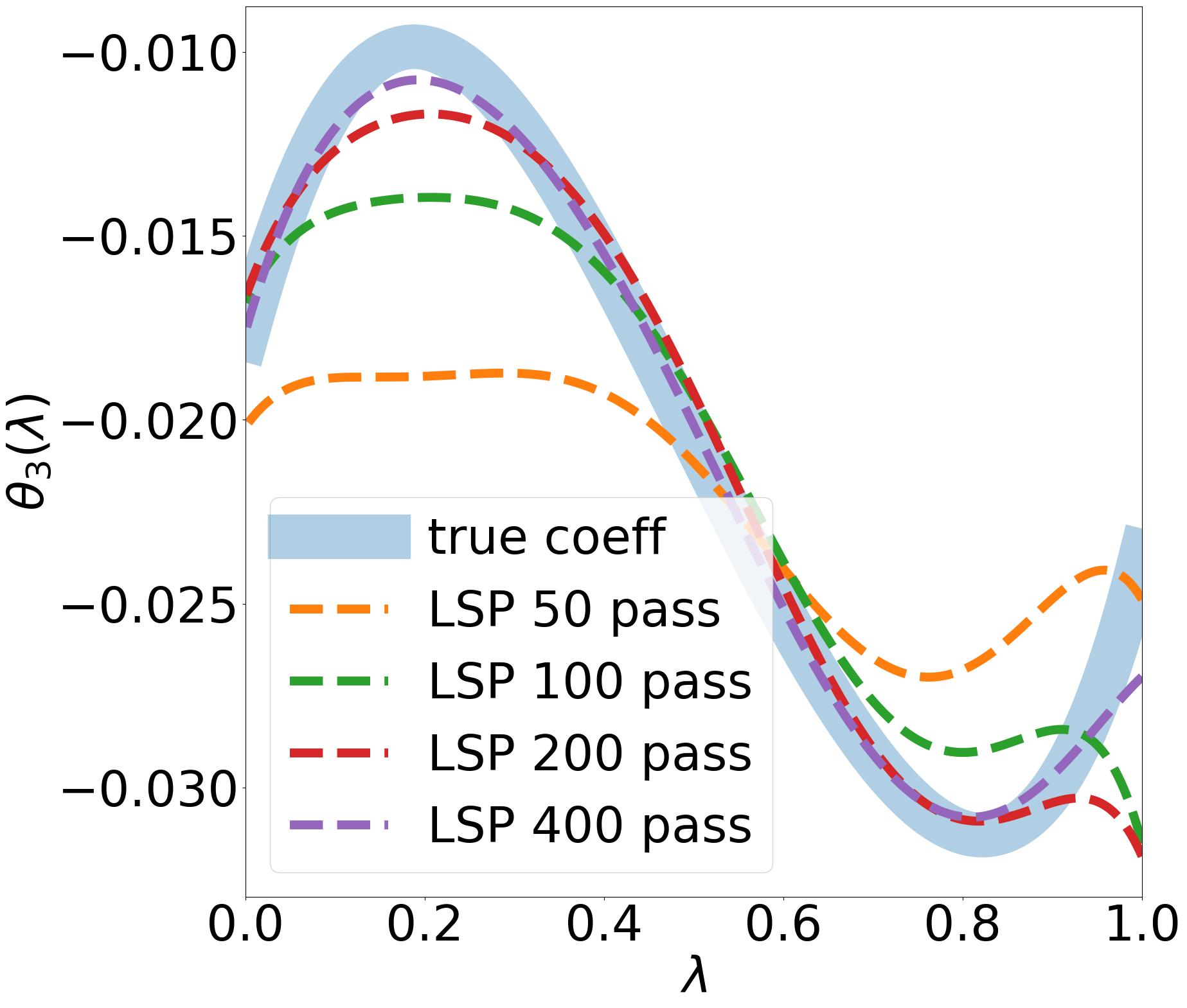}
         \caption{}
         \label{fig:CPP_exact_jacobi}
     \end{subfigure}

     \caption{{\bf LSP and ALSP for Weighted Binary Classification.} 
     (a) Compares methods using Jacobi polynomials.
     (b) Error in solution path as a function of $\lambda$ using Jacobi polynomials.
     (c) and (d) compare $\btheta_3(\lambda)$ across $\lambda$ in the Legendre and Jacobi bases respectively with $7$ polynomials.}
     
     % \centering
     %     \includegraphics[width=.35\textwidth]{CPP_exact_3_legendre.png}
     %     \caption{{\bf LSP: Interpolating the Solution Path for Weighted Binary Classification.} Compare $\bPhi_{3, 1:7}(\lambda)\betahat$ to $\btheta^\ast_3(\lambda)$ across $\lambda$ after running LSP with $7$ Legendre polynomials for 500, 100, 200, 400 gradient calls.}
        \label{fig:exact}
\end{figure}

\Cref{fig:preview} in the introduction and \cref{fig:exact} panel (a) show the performance for the Legendre and Jacobi bases respectively.  
As we can see, LSP methods converge rapidly to an irreducible error, and this error decreases as we add basis functions.  Our adaptive method converges rapidly.  
This performance is resonated by \cref{fig:exact} panels (c) and (d), where we 
plot the coefficient profile for $\theta^*_3(\lambda)$ when using $7$ Legendre and Jacobi polynomials.  The approximation improves as we increase the number of iterations, and eventually become closely aligned with the true solution path.

\cref{fig:exact} panel (b) shows that the largest errors occur almost periodically at a few places across $\Lambda = [0,1]$, confirming the equioscillation theorem.

% \Cref{fig:preview} should be contrasted with \cref{fig:exact} panel (a) which uses the Laguerre polynomials in the alternate parameterization.  Because of the non-uniform spacing (in $w(\lambda)$), the discretization approach converges slowly.  Indeed one can see in panel (b) that the largest errors occur near $w(\lambda) = 0$ (corresponding to $\lambda$ large). 
%  By contrast, errors for the Legendre polynomials in the same plot are uniformly small, near the tolerance of the optimization algorithm and so appear as noise.
 
%  Moreover, because the exponential distribution favors small $\lambda$, the LSP methods with Laguerre polynomials have more difficulty approximating the solution path for large $\lambda$ than the Legendre polynomials and, hence, also converge more slowly.  This is most clearly seen by comparing panels (c) and (d) where we 
% plot the coefficient profile for $\theta^*_3(\lambda)$ when using $7$ polynomials.  The approximation improves for both methods as we increase the number of iterations, but in (d), the fit for large $\lambda$ is very poor.  This poor fit drives the slow convergence of the uniform error in panel (a).  

Practically, this experiment suggests that 
% $\mathbb P_\lambda$ should be chosen to favor the regions of $\Lambda$ that are most important to decision-making a priori, and taken to be uniform in the absence of information.  
 when the region of $\Lambda$ is compact, taking a polynomial basis paired with a distribution supported everywhere over $\Lambda$ is favorable.
The polynomial basis can then be taken to be orthogonal to $\mathbb P_\lambda$ as discussed in \cref{sec:Adaptive}.

% Though our theory suggests that a constant learning rate is sufficient to guarantee convergence of the last iterates $\bbeta_t$ for LSP in expectation, in practice, deploying a combination of weighted averaging sum of iterates $\bar\bbeta_t = (1-\frac{2}{t+2})\bbeta_{t-1} + \frac{2}{t+2}\bbeta_t$ (\citep{lacoste2012simpler} Section 3.2) and adaptively adjusting learning rate according to the Distance Diagnostic (\citet{pesme2020convergence} Section 4) help reduce stochastic uncertainty (\cref{fig:err_along_soln_path_jacobi}). With a constant learning rate, the error exhibits a random walk around the plateau value; by adjusting the learning rate, we can stabilize the error around the threshold. 
% We use ($r=0.97$, $q=1.3$, $k_0=5$, $thresh=0.6$) for the Legendre, and ($r=0.97$, $q=1.5$, $k_0=5$, $thresh=0.6$) for the Laguerre bases, which are very 
% similar to the parameters used in (\citet{pesme2020convergence} Section 4). 

% In \cref{fig:CPP_exact_legendre} and \cref{fig:CPP_exact_jacobi}, which consider the first 7 polynomials, we plot the coefficient profiles across $\Lambda \subseteq \bbR$ for weighted binary classification to directly observe the convergence behavior of LSP.  In both figures, we observe convergence of the coefficient profiles as more iterations are run and, we also see a gap between the simulated and true coefficients due to the mispecification error of the basis function class as predicted by our theoretical results. 

\subsection{Portfolio Allocation}
\label{sec:portfolioAlloc}
We next consider a portfolio allocation problem calibrated to real data where $\btheta \in \mathbb R^d$ represents the weights on $d=10$ different asset classes.  Namely, let $\mu \in \mathbb R^d$ and $\Sigma \in \mathbb R^{d \times d}$ be the mean and covariance matrix of the returns of the different asset classes. We fit these parameters to the monthly return data from Aug. 2014 to July 2024 using the Fama-French 10 Industry index dataset.\footnote{\url{https://mba.tuck.dartmouth.edu/pages/faculty/ken.french/data_library.html}}

We then solve 
\begin{equation*}
    \begin{aligned}
        \min_{\btheta}& \quad -\lambda_1\mu^\top \btheta + \lambda_2\btheta^\top \Sigma \btheta + \bar \ell_{1}(\btheta). 
    \end{aligned}
\end{equation*}
Here the first term represents the (negative) expected return, the second represents the risk (variance) and the third is a smoothed version $\ell_1$ regularization to induce sparsity.  Namely, 
$\bar \ell_{1}(\btheta) := \sum_{i = 1}^d \sqrt{\theta_i^2 + .01^2} - .01.$
The parameters $\lambda_1 \in [0, 1]$ and $\lambda_2 \in [0.2, 1]$ control the tradeoffs in these multiple objectives.

% In a more complicated $\Lambda \subseteq \bbR^2$ setting, we investigate a portfolio allocation problem.
%     Let $\theta = (\theta_1, \ldots, \theta_d)$ be weights on $d$ different assets. We use the monthly return data from Aug 2014 to July 2024 on 10 industries  (). Let $\mu$ and $\Sigma$ be the mean vector and covariance matrix of monthly returns of the different assets. $l_{1 \text{, smooth}}(\theta) := \sum_{i = 1}^d \sqrt{\theta_i^2 + .01^2} - .01$ denote a smoothened $l_1$-norm regulation. For $\lambda_1\in [0, 1],\ \lambda_2 \in [.2, 1]$, we consider the optimization problem that balances expected return, risk (variance), and transaction cost ($l_1$ regulation):
%     \begin{equation*}
%     \begin{aligned}
%         \min_{\theta}& \quad -\lambda_1\mu^\top \theta + \lambda_2\theta^\top \Sigma \theta + l_{1 \text{, smooth}}(\theta) . 
%         % \\
%         % \text{s.t.}& \quad \sum_{i = 1}^d \theta_i = 1.
%     \end{aligned}
%     \end{equation*}

Ground truth $\btheta^*(\lambda)$ is computed over a $100\times100$ grid.

We focus on bivariate-Legendre polynomials for our basis, scaled and shifted to be orthogonal to uniform distribution on $[0,1] \times [.2,1]$.  We initialize with ALSP with 2 polynomials in each dimension, and iterate by adding one polynomial in each dimension, so that $q = 4, 9, 16, 25$, after which we stop.

Unlike our previous experiment, to showcase that the qualitative insights of our theory hold for other algorithms other than SGD, we use \texttt{torch.optim.LBFGS} for LSP, ALSP and uniform discretization (for a fair comparison).  Unlike SGD, L-BFGS uses both function and gradient evaluations.  We restrict it to use only $10$ function calls per gradient step so that the total work is still proportional to the number of gradient calls.

% and thus, we record the computation cost as the number of function objective evaluations. 

% The true optimums are computed with the \texttt{scipy.optimize.minimize} trust-region solver over a $100 \times 100$ grid. 
% Similar to the weighted binary classification problem, we approximate their resultant solution path error $\epsilonsp(\cdot)$ by comparing to this fine grid of true optimum, and we record the total computation cost across all grid points for the uniform discretization scheme. However, unlike the previous problem, we use \texttt{torch.optim.BFGS} for both LSP and uniform discretization, and thus, we record the computation cost as the number of function objective evaluations.

\begin{figure}[hbt!] 
\begin{center}
\includegraphics[width=.4\textwidth]{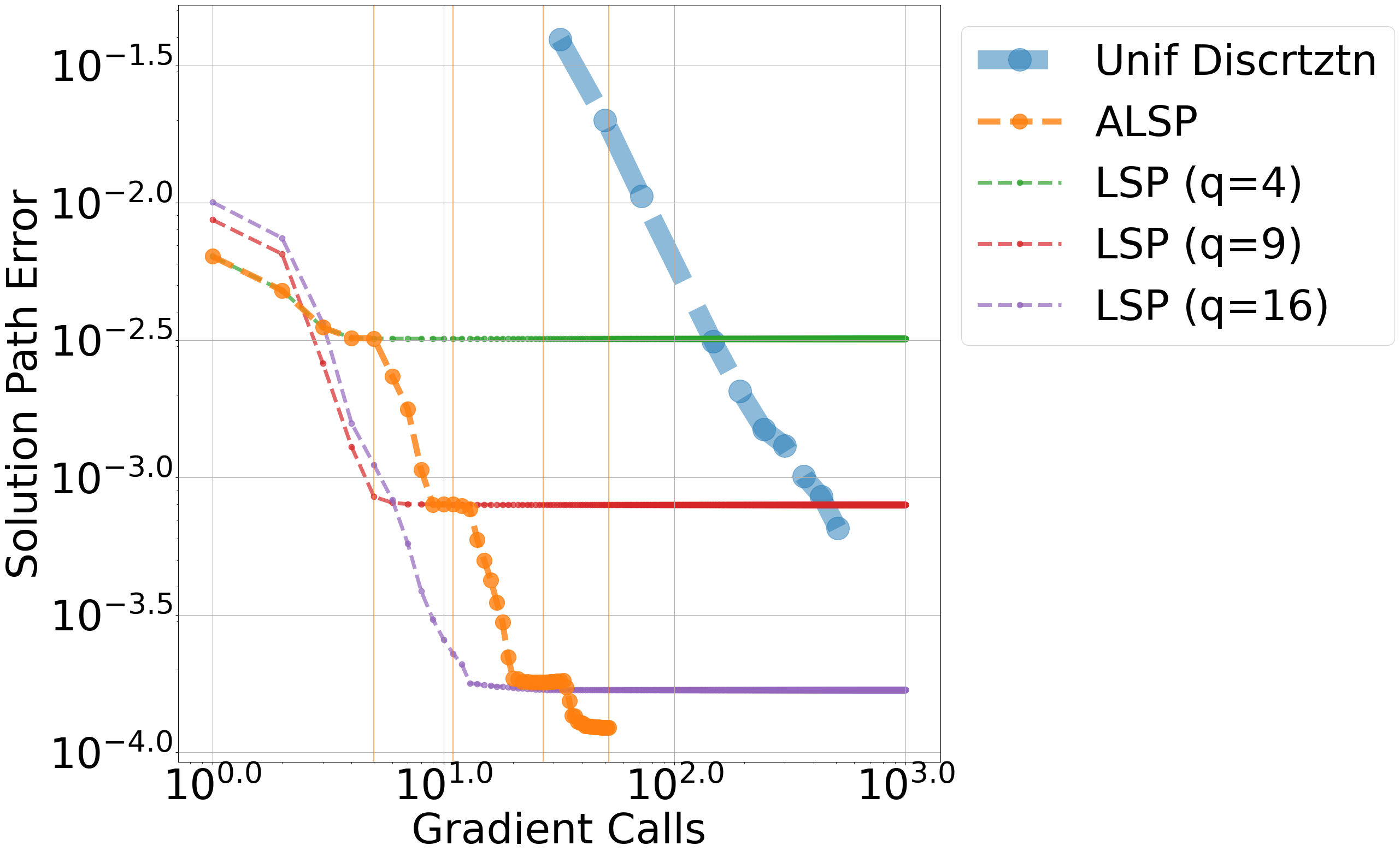}
\end{center}
\caption{{\bf LSP and ALSP for Portfolio Allocation.} Compares methods using the first $q = 4, 9, 16$ bivariate-Legendre polynomials.
\label{fig:boosted_and_fixed_exact_2d}}
\end{figure}

As in our previous experiment, we see that LSP rapidly converges to an irreducible error and then plateaus.  By contrast, ALSP seems to make continued progress as we add polynomials. 
Both methods substantively outperform uniform discretization.

The improved performance over uniform discretization is partially attributable to the increased dimension of $\Lambda$ (because discretization suffers from the curse of dimensionality), but is also because (traditionally) discretization interpolates solutions in a piecewise constant fashion.  In this example, $\btheta^*(\lambda)$ is very smooth. See \cref{fig:true_soln_path_2d} in Appendix.  Hence, polynomial can learn to interpolate values very fast, while uniform discretization needs a great deal more resolution.

{
\subsection{Portfolio Allocation with Moderate Dimensional $\Lambda$}\label{sec:highDim}

We next modify the portfolio allocation problem in \cref{sec:portfolioAlloc} to include transaction costs relative to the current portfolio as follows:  
\[
h(\theta, \lambda) \ = \ 
-\lambda_1 \cdot \mu^\top \theta + \lambda_2 \cdot \theta^\top \Sigma \theta + \|\theta - \lambda_{3:12}\|_2^2.
\]
Here $\lambda_{3:12} \in  \mathbb R^{10}$ represents the current portfolio holding, and $\lambda \in \mathbb R^{12}$.
We use the following basis
\[
\lambda_{(j \text{ mod } 12)} \cdot \lambda_2^{\lfloor j/12 \rfloor}, \ \ j=1, \ldots, q,
\]
for each component of $\thetahat_i(\cdot)$, for $i=1, \ldots, 10$, and let $q = 12, 24, 36$ in our experiments.  
% In other words, we choose $\bPhi(\lambda) \in \mathbb R^{10 \times 10q}$ to be block-diagonal with $10$ blocks of size $1 \times q$. Denote this block by $\psi(\lambda) \in \mathbb R^{1 \times q}$, and the elements of $\psi(\lambda)$ are
% \[
% \psi_j (\lambda) = \lambda_{(j\bmod 12)} \cdot \lambda_2^{\lfloor{j / 12}\rfloor}
% \qquad
% \text{for } j = 1, ..., q.
% \]
As $q\rightarrow \infty$, this basis contains the optimal solution path in its span for any choice of $\mu, \Sigma$ (c.f. \cref{sec:taylorExpansion}).  

For simplicity, we ran LSP using \texttt{torch.optim.LBFGS} to solve an ERM version of the problem in Step 6 of \cref{alg:adaptiveSgd} on a training set of size $1000$, and present results for an independent validation set of size $1000$.  
% Please see \url{https://ibb.co/5s5Rszr} for performance on the training set by iteration of L-BFGS, and \url{https://ibb.co/MNdmDjQ} for performance on the validation set.

Discretization on a $12$-dimensional grid is computationally challenging.  Hence, to assess our method we instead compute an ERM approximation to $\Eb{h(\btheta^*(\lambdarv), \lambdarv)}$ over the training/validation sets, i.e., compute $\frac{1}{1000} \sum_{i=1}^{1000} h(\btheta^*(\lambda_i), \lambda_i)$ by solving $1000$ separate optimization problems.  This value is the dotted grey line in \cref{fig:erm}.  The smaller the gap, the closer $\btheta^*(\lambda_i)$ is to the span of the basis.

Similar to previous experiments, as the basis size increases, performance improves, and even with a fairly small basis, the error is quite small.  
%While not a completely exhaustive experiment, we hope this gives some evidence that our method scales reasonably in $\text{dim}(\Lambda)$.  

\begin{figure}
     \centering
          \begin{subfigure}[b]{0.2375\textwidth}
         \centering
         \includegraphics[width=\textwidth]{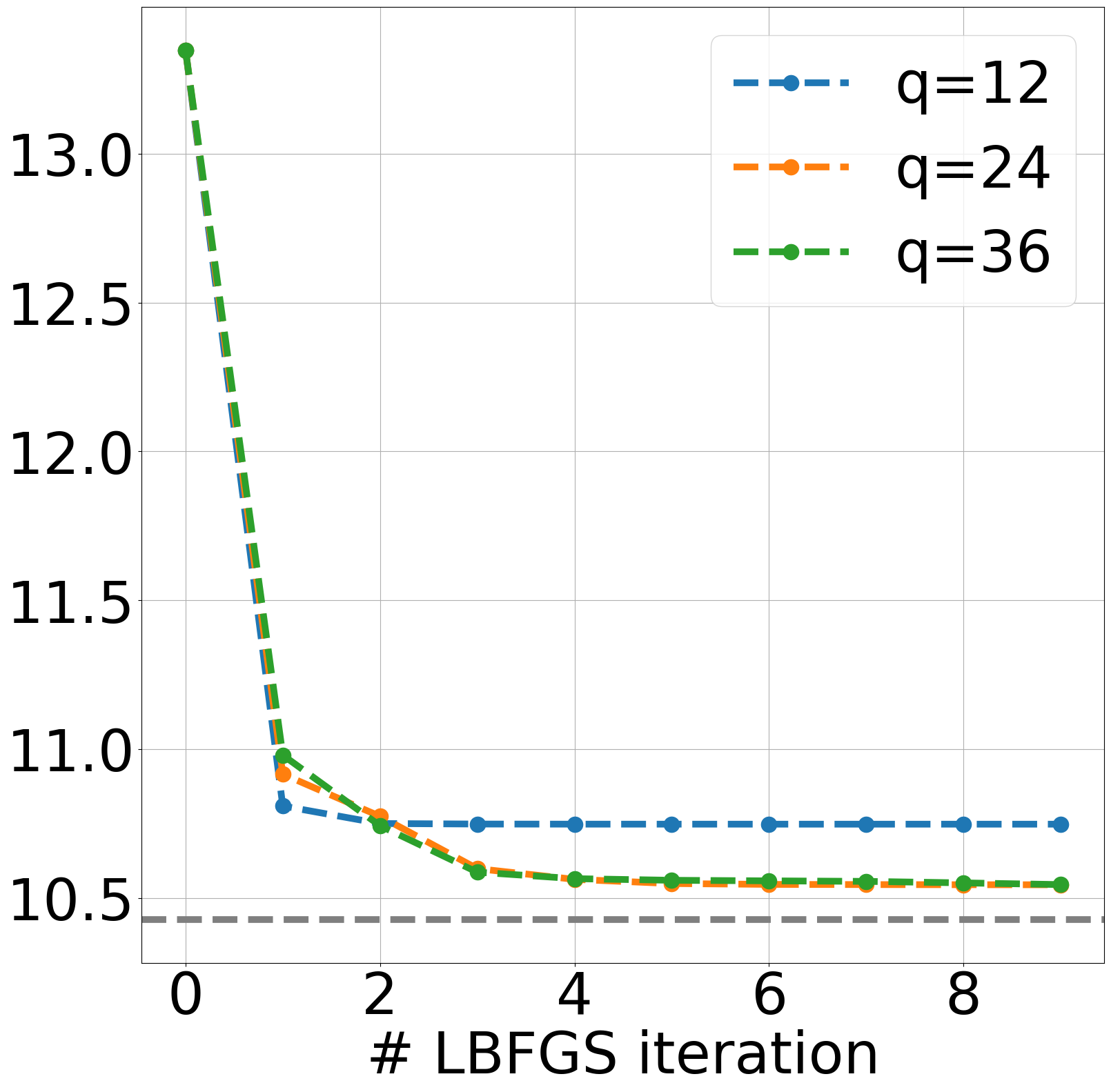}
         \caption{}
         \label{fig:high_dim_train}
     \end{subfigure}
     \hfill
    \begin{subfigure}[b]{0.2375\textwidth}
         \centering
         \includegraphics[width=\textwidth]{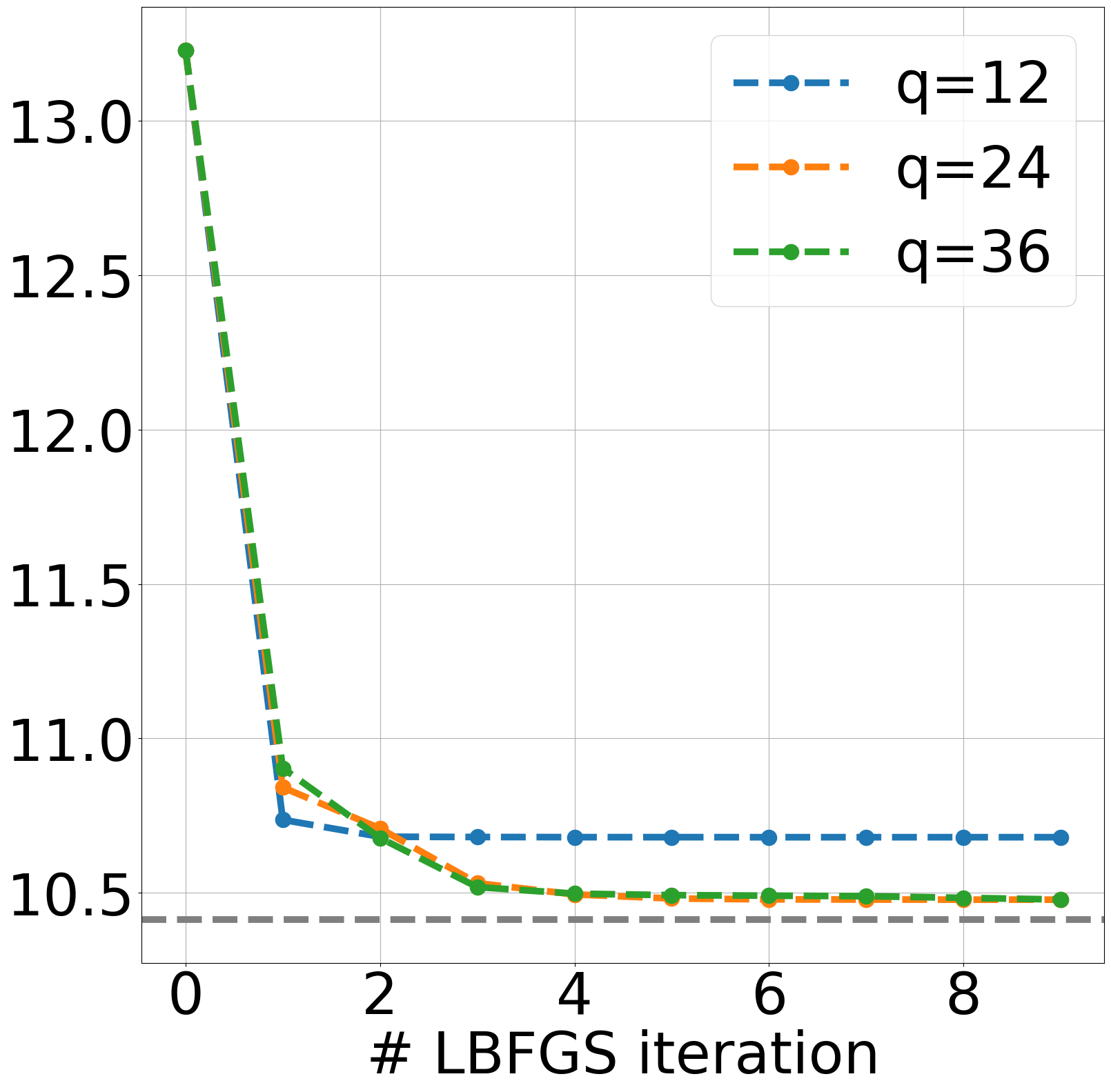}
         \caption{}
         \label{fig:high_dim_val}
     \end{subfigure}

     \caption{{\bf Portfolio Allocation,  $\dim(\Lambda) = 10$.} 
    Comparing the ERM objective for varying $q$.  Panel a) training, Panel b) validation.
    }
    \label{fig:erm}
\end{figure}
}

%!TEX root = 0_Main.tex

\section{CONCLUSION}

We propose a new method for learning the optimal solution path of a family of problems by reframing it as a single stochastic optimization problem over a linear combination of pre-specified basis functions. Compared to discretization schemes, our method offers flexibility and scalability by taking a global perspective on the solution path.  
We prove that our problem satisfies a certain relaxed weak-growth condition that allows us to solve the single optimization problem very efficiently when using sufficiently rich bases. Theoretical results in special cases and numerical experiments in more general settings support these findings.
Future research might more carefully examine the interplay between the parameterization of the family of problems and the choice of basis. One might also consider other universal function approximators (e.g. deep neural networks, forests of trees) within this context.

%%%%%%%%%%%%%%%%%%%%%%%%%%%%%%%%%%%%%%%%%%%%%%%%%%%%%%%%%%%%

%\section*{References}

\newpage
{
\subsubsection*{Acknowledgements}
PG acknowledges the support of the NSF AI Institute for Advances in Optimization, Award 2112533. The authors thank the anonymous reviewers for their thoughtful and constructive comments.
}

\bibliographystyle{abbrvnat}
\bibliography{refs}

% \begin{thebibliography}{}
% \setlength{\itemindent}{-\leftmargin}
% \makeatletter\renewcommand{\@biblabel}[1]{}\makeatother
% \bibitem{} J.~Alspector, B.~Gupta, and R.~B.~Allen (1989).
%     \newblock Performance of a stochastic learning microchip.
%     \newblock In D. S. Touretzky (ed.),
%     \textit{Advances in Neural Information Processing Systems 1}, 748--760.
%     San Mateo, Calif.: Morgan Kaufmann.

% \bibitem{} F.~Rosenblatt (1962).
%     \newblock \textit{Principles of Neurodynamics.}
%     \newblock Washington, D.C.: Spartan Books.

% \bibitem{} G.~Tesauro (1989).
%     \newblock Neurogammon wins computer Olympiad.
%     \newblock \textit{Neural Computation} \textbf{1}(3):321--323.
% \end{thebibliography}

%%%%%%%%%%%%%%%%%%%%%%%%%%%%%%%%%%%%%%%%%%%%%%%%%%%%%%%%%%%%

\newpage
\onecolumn
\appendix

\section*{Checklist}
% %%% BEGIN INSTRUCTIONS %%%
% The checklist follows the references. For each question, choose your answer from the three possible options: Yes, No, Not Applicable.  You are encouraged to include a justification to your answer, either by referencing the appropriate section of your paper or providing a brief inline description (1-2 sentences). 
% Please do not modify the questions.  Note that the Checklist section does not count towards the page limit. Not including the checklist in the first submission won't result in desk rejection, although in such case we will ask you to upload it during the author response period and include it in camera ready (if accepted).

% \textbf{In your paper, please delete this instructions block and only keep the Checklist section heading above along with the questions/answers below.}
% % %%% END INSTRUCTIONS %%%

 \begin{enumerate}

 \item For all models and algorithms presented, check if you include:
 \begin{enumerate}
   \item A clear description of the mathematical setting, assumptions, algorithm, and/or model. [Yes]
   \item An analysis of the properties and complexity (time, space, sample size) of any algorithm. [Yes]
   \item (Optional) Anonymized source code, with specification of all dependencies, including external libraries. [Yes]
 \end{enumerate}

 \item For any theoretical claim, check if you include:
 \begin{enumerate}
   \item Statements of the full set of assumptions of all theoretical results. [Yes]
   \item Complete proofs of all theoretical results. [Yes]
   \item Clear explanations of any assumptions. [Yes]     
 \end{enumerate}

 \item For all figures and tables that present empirical results, check if you include:
 \begin{enumerate}
   \item The code, data, and instructions needed to reproduce the main experimental results (either in the supplemental material or as a URL). [Yes]
   \item All the training details (e.g., data splits, hyperparameters, how they were chosen). [Yes]
         \item A clear definition of the specific measure or statistics and error bars (e.g., with respect to the random seed after running experiments multiple times). [Not Applicable] 
         \item A description of the computing infrastructure used. (e.g., type of GPUs, internal cluster, or cloud provider). [Not Applicable]
 \end{enumerate}

 \item If you are using existing assets (e.g., code, data, models) or curating/releasing new assets, check if you include:
 \begin{enumerate}
   \item Citations of the creator If your work uses existing assets. [Yes]
   \item The license information of the assets, if applicable. [Yes]
   \item New assets either in the supplemental material or as a URL, if applicable. [Not Applicable]
   \item Information about consent from data providers/curators. [Not Applicable]
   \item Discussion of sensible content if applicable, e.g., personally identifiable information or offensive content. [Not Applicable]
 \end{enumerate}

 \item If you used crowdsourcing or conducted research with human subjects, check if you include:
 \begin{enumerate}
   \item The full text of instructions given to participants and screenshots. [Not Applicable]
   \item Descriptions of potential participant risks, with links to Institutional Review Board (IRB) approvals if applicable. [Not Applicable]
   \item The estimated hourly wage paid to participants and the total amount spent on participant compensation. [Not Applicable]
 \end{enumerate}

 \end{enumerate}

\newpage
%!TEX root = 0_Main.tex

% \input{Picking Stepsize}

% \input{Noisy Gradients}
\section{Omitted Proofs}

\subsection{Proofs from \cref{sec:Model}}

\begin{proof}[Proof for \Cref{lma:cPositive}.]
    Suppose by contradiction that $c = 0$.  Then there exists a corresponding eigenvector $\bv \in \bbR^p$ with $\bv \neq \boldsymbol 0$ such that
    \[
    \bbE_{\lambdarv\sim \bbP_\lambda}\left[ \bPhi(\lambdarv)^\top \bPhi(\lambdarv) \right] \bv  =  \boldsymbol 0 \implies
%    \bbE_{\lambdarv}\left[ \bv^\top \bPhi(\lambdarv)^\top \bPhi(\lambdarv) \bv \right]
%     \ = \     
\bbE_{\lambdarv\sim \bbP_\lambda}\left[ \|  \bPhi(\lambdarv) \bv \|^2_2 \right] = 0.
    \]
This further implies that $\bPhi(\lambdarv) \bv = 0$  almost everywhere. By the linear independence assumption in \cref{assum:components}, we then must have $\bv = \boldsymbol 0$, but this is a contradiction since $\bv$ is an eigenvector.  

For the second statement, notice
\[
c 
\leq 
\sigma_{\max}\left\{ \mathbb E_{\lambdarv\sim \bbP_\lambda}\left[ \bPhi(\lambdarv)^\top \bPhi(\lambdarv) \right] \right\}
\leq 
\mathbb E_{\lambdarv\sim \bbP_\lambda} \left[ \sigma_{\max} \left\{ \bPhi(\lambdarv)^\top \bPhi(\lambdarv) \right\}\right]
\leq 
C,
\]
where the penultimate inequality follows from Jensen's inequality and the convexity of the maximal eigenvalue function.
\end{proof}

Next, recall the family of functions $f(\cdot, \lambda)$ and the function $F(\cdot)$ defined in \eqref{eq:auxiliary}: 
\begin{equation*}
    f(\bbeta, \lambda) ~:=~ h(\bPhi(\lambda)\bbeta, \lambda), \ \ \ \
    F(\bbeta) ~:=~ \bbE_{\lambdarv \sim \bbP_\lambda}\left[f(\bbeta, \lambdarv)\right] = \bbE_{\lambdarv\sim \bbP_\lambda}\left[h(\bPhi(\lambdarv)\bbeta, \lambdarv)\right].
\end{equation*}
Before proceeding, we establish the smoothness and strong convexity of these functions.

\begin{proposition}[{\bf Uniform Smoothness and Strong Convexity}]\label{prop:regularity}
The family of functions $f(\cdot, \lambda)$, over all $\lambda \in \Lambda$, is $CL$-smooth. Moreover, $F(\cdot)$ is $CL$-smooth and $c\mu$-strongly convex.
\end{proposition}
\begin{proof}[Proof for \Cref{prop:regularity}.]
    Fix an arbitrary $\lambda\in\Lambda$. Then, for any $\bbeta$, $\bar{\bbeta} \in \bbR^p$, smoothness of $f(\cdot, \lambda)$ is certified by
    \begin{align*}
        \|\nabla f(\bbeta, \lambda) - \nabla f(\bar{\bbeta}, \lambda)\|
        & = \|\nabla_{\bbeta} h(\bPhi(\lambda)\bbeta, \lambda) - \nabla_{\bbeta} h(\bPhi(\lambda)\bar{\bbeta}, \lambda)\| 
    \\
        &\ \leq \|\bPhi(\lambda)^\top (\nabla_{\btheta} h(\bPhi(\lambda)\bbeta, \lambda) - \nabla_{\btheta} h(\bPhi(\lambda)\bar{\bbeta}, \lambda))\| &&(\text{chain rule})
    \\
        &\ \leq \|\bPhi(\lambda)^\top\|\cdot L\|\bPhi(\lambda)(\bbeta - \bar{\bbeta})\| &&(\text{$L$-smoothness})
    \\
        &\ \leq CL\|\bbeta - \bar{\bbeta}\| 
    \end{align*}
    for any $\bbeta$, $\bar{\bbeta} \in \bbR^d$. To obtain the same conclusions for $F(\cdot)$, we may take expectation over $\lambda \sim \bbP_\lambda$ on the above inequalities, and invoke the linearity of expectations as well as the property that
    \[
    \|\Eb{g(\cdot)}\| 
    \ \leq \  
    \Eb{\|g(\cdot)\|}.
    \]
    Next, we verify the strong convexity of $F(\cdot)$. Fixing $\bbeta$, $\bar{\bbeta} \in \bbR^d$, for any $\lambda\in\Lambda$, define $\btheta(\lambda): = \bPhi(\lambda)\bbeta$, $\bar{\btheta}(\lambda): = \bPhi(\lambda)\bar{\bbeta}$. Then, by strong convexity of $h$, we have
    \[
    h(\bar{\btheta}(\lambda), \lambda)
    \geq 
    h(\btheta(\lambda),\lambda) + \nabla_{\btheta} h(\btheta(\lambda), \lambda)^\top \left( \bar{\btheta}(\lambda) - \btheta(\lambda)\right) + \frac{\mu}{2}\|\bar{\btheta}(\lambda) - \btheta(\lambda)\|^2.
    \]
    Take expectation w.r.t. $\tilde\lambda \sim \bbP_\lambda$ on both sides, we obtain
    \begin{align*}
        F(\bar{\bbeta}) 
        &= \bbE\left[h(\bPhi(\lambdarv)\bar{\bbeta}, \lambdarv)\right]
    \\
        &\ \geq \bbE\left[h(\bPhi(\lambdarv)\bbeta, \lambdarv)\right] + \bbE\left[\nabla_{\btheta} h(\btheta(\lambdarv), \lambdarv)^\top \left(\bPhi(\lambdarv)(\bar{\bbeta} - \bbeta)\right)\right] + \frac{\mu}{2}\bbE\left[\|\bPhi(\lambdarv)(\bar{\bbeta} - \bbeta)\|^2\right]
    \\
        &\ =
        F(\bbeta) + \bbE\left[\bPhi(\lambdarv)^\top \nabla_{\btheta} h(\btheta(\lambdarv), \lambdarv)\right]^\top (\bar{\bbeta} - \bbeta) + \frac{\mu}{2}(\bar{\bbeta} - \bbeta)^\top \bbE\left[\bPhi(\lambdarv)^\top \bPhi(\lambdarv)\right](\bar{\bbeta} - \bbeta)
    \\
        &\ \geq  F(\bbeta) + \nabla F(\bbeta)^\top (\bar{\bbeta} - \bbeta) + \frac{c\mu}{2}(\bar{\bbeta} - \bbeta)^\top (\bar{\bbeta} - \bbeta)
    \end{align*}
    The second term in the last inequality comes from the Leibniz integral rule and the third term invokes a well-known property of the smallest eigenvalue of a positive definite matrix.
\end{proof}

\begin{proof}[{\bf Proof for \Cref{thm:decompose}.}]\label{pf:thmDecompose}

    % First, fix an arbitrary $\bar\epsilon>0$, we can pick a $\betasp \in \bbR^p$ s.t.
    % \[
    %      \epsilonsp(\betasp) < \epsilonsp^\ast + \bar\epsilon.
    % \]
    First, we prove that $\epsilonsp^\ast$ is attained at some $\betasp \in \bbR^p$.
    Observe that $\epsilonsp(\cdot)$ is continuous. Thus, its level set $M(1) := \{\bbeta\in \bbR^p: \epsilonsp(\bbeta) \leq 1\}$ is closed. 
    Furthermore, $M(1)$ is bounded. To see this, 
    % fix an arbitrary $\lambda \in \Lambda$, then there is some $\bbeta_\lambda$ s.t. $\bPhi(\lambda)\bbeta_\lambda = \btheta^*(\lambda)$. By definition, $\bbeta_\lambda \in M(1)$.
    % Now 
    for any $\bbeta\in M(1)$, by strong convexity of $F$,
    \begin{align*}
        \frac{c\mu}{2}\|\bbeta - \betaavg\|^2
    & 
    ~\leq~
        F(\bbeta) - F(\betaavg)\\
    & 
    ~=~
        \bbE_{\lambdarv\sim \bbP_\lambda}\left[h(\bPhi(\lambdarv)\bbeta, \lambdarv) - h(\bPhi(\lambdarv)\betaavg, \lambdarv)\right]\\
    & 
    ~\leq~ 
        \bbE_{\lambdarv\sim \bbP_\lambda}\left[h(\bPhi(\lambdarv)\bbeta, \lambdarv) - h(\btheta^*(\lambdarv), \lambdarv)\right]\\
    &
    ~\leq~
        \epsilonsp(\bbeta)\tag{\theequation}\label{usefulEq}\\
    & ~\leq~ 
        1.
    \end{align*}
    Since $\bbeta$ is arbitrary, for any $\bbeta$, $\bar{\bbeta} \in M(1)$,
    \[
        \|\bbeta - \bar{\bbeta}\|
    ~\leq~ 
        \|\bbeta - \betaavg\| + \|\bar{\bbeta} - \betaavg\|
    ~\leq~ 
        \frac{2\sqrt{2}}{\sqrt{c\mu}}.
    \]
    Applying the Weierstrass Theorem, there exists $\betasp \in M(1) \subset \bbR^p$ s.t. 
    \[
        \epsilonsp(\betasp) 
    ~=~
        \epsilonsp^\ast.
    \]

    Next, we prove the main result of the theorem.
    Again, fix an arbitrary $\lambda$. By $L$-smoothness of $h(\cdot, \lambda)$,
    \begin{equation}\label{diffOfh}
        h(\bPhi(\lambda) \bbeta, \lambda) - h(\btheta^*(\lambda), \lambda)
    \ \leq\ 
        \frac{L}{2}\|\bPhi(\lambda)\bbeta - \btheta^*(\lambda)\|^2
    \end{equation}
    Using the identity 
    \(
        (a + b + c)^2 \leq 4a^2 + 4b^2 + 4c^2
    \)
    and triangle inequality on the right-hand side, this inequality becomes
    \[
        h(\bPhi(\lambda) \bbeta, \lambda) - h(\btheta^*(\lambda), \lambda)
    \ \leq\ 
        2L(\|\bPhi(\lambda) (\bbeta - \betaavg)\|^2 + \|\bPhi(\lambda) (\betaavg - \betasp)\|^2 + \|\bPhi(\lambda) \betasp - \btheta^*(\lambda))\|^2).
    \]
% \vg{@Qiran:  I suspect this step of the proof is loose.  The reason is that we first go from a difference of $h$'s in \cref{diffOfh} to a norm using $L$-smoothness, and then we go back from a norm difference to difference of $h$ in the step right below this sentence.  I suspect what we need is some generalization of the identity I wrote to strongly convex functions.}    
    By strong convexity of $h(\cdot, \lambda)$ and definition of $\betasp$,
    \[
        \|\bPhi(\lambda) \betasp - \btheta^*(\lambda)\|^2
    \ \leq\ 
        \frac{2}{\mu} \sup_{\lambda\in\Lambda}\left\{h(\bPhi(\lambda)\betasp, \lambda) - h(\btheta^*(\lambda), \lambda)\right\}
    \ =\ 
        \frac{2}{\mu}\epsilonsp(\betasp)
    \ =\  
        \frac{2}{\mu}\epsilonsp^\ast,
    \]
    So
    \begin{equation}\label{ineq:thm1MidStep}
        \epsilonsp(\bbeta)
    \ \leq\ 
        2CL(\|\bbeta - \betaavg\|^2 + \|\betaavg - \betasp\|^2) + \frac{4L}{\mu}\epsilonsp^\ast.
    \end{equation}
    
    Next, we bound $\|\betaavg - \betasp\|^2$. Observe that due to the optimality of $\btheta^*(\lambda)$ for each $\lambda$,
    \[
        \Eb{h(\bPhi(\lambdarv)\betasp, \lambdarv)} - \Eb{h(\bPhi(\lambdarv)\betaavg, \lambdarv)}
        \ \leq\ \Eb{h(\bPhi(\lambdarv)\betasp, \lambdarv) - h(\btheta^*(\lambdarv), \lambdarv)}
        \ \leq\ 
        \epsilonsp^\ast.
    \]
    On the other hand, from the optimality of $\betaavg$ and strong-convexity of $F(\cdot)$,
    \[
        F(\betasp) - F(\betaavg)
        \ = \ 
        \Eb{h(\bPhi(\lambdarv)\betasp, \lambdarv)} - \Eb{h(\bPhi(\lambdarv)\betaavg, \lambdarv)} 
        \ \geq\ 
        \frac{c\mu}{2} \| \betaavg - \betasp\|^2.
    \]
    
    Combining this inequality with the previous one shows
    \(
        \frac{c\mu}{2} \| \betaavg - \betasp\|^2
    \ \leq \
        \epsilonsp^\ast,
    \)
    which implies that
    \[
        \| \betaavg - \betasp\|^2  \ \leq \ \frac{2\epsilonsp^\ast}{c\mu}.
    \]
    Substitute the above into \Cref{ineq:thm1MidStep}. Take expectations over sample paths on both sides, we conclude that
    \[
        \epsilonsp(\bbeta)
    \ \leq\ 
        2CL\left(\|\bbeta - \betaavg\|^2 + \frac{2\epsilonsp^\ast}{c\mu}\right) + \frac{4L}{\mu}\epsilonsp^\ast.
    \]
    Rearranging and using the fact that $C/c \geq 1$ verifies the statement.
\end{proof}

\subsection{Proofs from \cref{sec:convGuaranteeSGD}}
\begin{proof}[Proof for \Cref{lma:convergeWGC}]
The proof follows the proof of Theorem 3.1 of \citet{gower2019sgd}; we include it for completeness under the assumption of the RWGC \cref{def:rwgc}.
Let $\eta := \frac{\bar\eta}{\rho}$ denote the step-size, $\bw^*:= \arg\min_{\bw \in \bbR^d} G(\bw)$ denote the optimality input of $G$. For $t \geq 0$, we have
\begin{align*}
\|\bw_{t+1} - \bw^*\|^2 &= \|\bw_{t} - \bw^* - \eta\nabla g(\bw_t, z_t)\|^2 \\
&= \|\bw_{t} - \bw^*\|^2 -2\eta\nabla g(\bw_t, z_t)^\top (\bw_{t} - \bw^*) + \eta^2\|\nabla g(\bw_t, z_t)\|^2.
\end{align*}
Taking conditional expectations yields
\begin{align}\notag
\bbE\left[\|\bw_{t+1} - \bw^*\|^2 \;\middle|\; \bw_t\right] 
~&=~ \|\bw_{t} - \bw^*\|^2 - 2\eta\nabla G(\bw_t)^\top (\bw_t - \bw^*) + \eta^2\bbE\left[\|\nabla g(\bw_t, z_t)\|^2 \;\middle|\; \bw_t\right] 
\\ \notag
~&\leq~ \|\bw_{t} - \bw^*\|^2 - 2\eta\nabla G(\bw_t)^\top (\bw_t - \bw^*) + \eta^2[2\rho(G(\bw_t) - G^*) + \sigma^2] 
\\ \notag
~&\leq~ \|\bw_{t} - \bw^*\|^2 + 2\eta\left(G^* - G(\bw_t) - \tfrac{\mu_g}{2}\|\bw_t - \bw^*\|^2\right) + \eta^2[2\rho (G(\bw_t) - G^*) + \sigma^2] 
\\ \notag
~&=~ \left(1 - \eta\mu_g\right)\|\bw_t - \bw^*\|^2 + 2\left(\eta^2\rho -\eta\right)(G(\bw_t) - G^*) + \eta^2\sigma^2 \\ \label{eq:usingRWGC}
~&\leq~ \left(1 - \eta\mu_g\right)\|\bw_t - \bw^*\|^2 + \eta^2\sigma^2
% \\ \notag
% &\leq \left(1 - \eta\mu_g\right)\|\bw_t - \bw^*\|^2 + \eta^2\sigma^2.\notag
\end{align}
The first inequality above uses the RWGC and the second uses strong convexity. The final inequality holds since $\eta^2\rho - \eta = \left(\frac{\bar\eta}{\rho}\right)^2\rho - \frac{\bar\eta}{\rho} = \left(\frac{\bar\eta}{\rho}\right)\left(\bar\eta - 1\right)$ and $\bar\eta \leq 1$. Furthermore, since $\bar\eta \leq \frac{\rho}{\mu_g}$, we have $\eta\mu_g = \left(\frac{\bar\eta}{\rho}\right)\mu_g \leq 1$; therefore, 
% , and the third uses $\eta \in \left[0, \min\left\{\frac{1}{\rho L_g}, \frac{1}{\mu_g}\right\}\right]$. 
% Taking overall expectations and iterating the inequalities yields
% \begin{align*}
% \bbE[\|\bw_{t} - \bw^*\|^2] &\leq (1 - \eta\mu_g)^t\|\bw_{0} - \bw^*\|^2 + \eta^2\sigma^2\sum_{k = 0}^{t-1} (1 - \eta\mu_g)^k \\
% &= (1 - \eta\mu_g)^t\|\bw_{0} - \bw^*\|^2 + \frac{\eta^2\sigma^2(1 - (\eta\mu_g)^t)}{1 - \eta\mu_g} \\
% &\leq (1 - \eta\mu_g)^t\|\bw_{0} - \bw^*\|^2 + \frac{\eta^2\sigma^2}{1 - \eta\mu_g}
% \end{align*}
% \end{proof}
% % \vg{@Qiran/@Paul: I think this is a stronger version of Lemma 4.4.  I don't know if it's helpful.  Can you please check?  I wrote it quickly and might have made an error.}
% \begin{proof}[A stronger Version of Lemma 4.4]
% We follow the proof of Lemma 4.4 up until \cref{eq:usingRWGC}.  In the original proof, we use the fact that $\eta \rho L_g - 1 < 0$ to drop the second term.  In this analysis we will keep this term.  
by iterating \cref{eq:usingRWGC} and taking overall expectations, we get
\begin{align*}
\Eb{\| \bw_t - \bw^*\|^2} ~&\leq~ (1-\eta \mu_g)^t \| \bw_0 - \bw^* \|^2 + \eta^2\sigma^2\sum_{k = 0}^{t-1}(1 - \eta\mu_g)^k \\
~&\leq~ (1-\eta \mu_g)^t \| \bw_0 - \bw^* \|^2 + \frac{\eta^2 \sigma^2}{1 - (1 - \eta \mu_g)} \\
~&=~ (1-\eta \mu_g)^t \| \bw_0 - \bw^* \|^2 + \frac{\eta \sigma^2}{\mu_g},
\end{align*}
where the second inequality uses the geometric series bound. Recalling that $\eta = \frac{\bar\eta}{\rho}$ completes the proof.

\end{proof}

\begin{proof}[Proof for \Cref{lma:satisfyWGC}.]
Recall that \cref{prop:regularity} already establishes the smoothness of $f$ and smoothness and strong convexity of $F$.

Define 
    \[
    \betabar(\lambda) \in \arg\min_{\bbeta\in\bbR^p} f(\bbeta, \lambda) = \arg\min_{\bbeta\in\bbR^p} h(\bPhi(\lambda) \bbeta, \lambda).
    \]
    % \paul{$f$ is not L-smooth. Should this proof be rewritten in terms of $h$? Doing it that way, we should be able to remove $\bar{\beta}(\lambda)$.}
    By $CL$-smoothness, we have
    \begin{equation*}
        f(\bbeta, \lambda) \geq f(\betabar(\lambda), \lambda) + \nabla f(\betabar(\lambda), \lambda)^\top(\bbeta - \betabar(\lambda)) + \frac{1}{2CL}\|\nabla f(\bbeta, \lambda) - \nabla f(\betabar(\lambda), \lambda)\|^2.
    \end{equation*}
    Using $\nabla f(\betabar(\lambda), \lambda) = 0$ for all $\lambda$, rearranging and recalling the definitions of $\betaavg$ (\cref{sec:Model}) and $\betasp$ (proof of \cref{thm:decompose}) yields
    \begin{align*}
        \frac{1}{2CL}\|\nabla f(\bbeta, \lambda)\|^2 ~&\leq~ f(\bbeta, \lambda) - f(\betabar(\lambda), \lambda) 
    \\
        ~&\leq~ f(\bbeta, \lambda) - h(\btheta^*(\lambda), \lambda)
    \\
        ~&=~ (f(\bbeta, \lambda) - f(\betaavg, \lambda)) + (f(\betaavg, \lambda) - f(\betasp, \lambda)) + (f(\betasp, \lambda) - h(\btheta^*(\lambda), \lambda))
    \\ 
        ~&\leq~ (f(\bbeta, \lambda) - f(\betaavg, \lambda)) + (f(\betaavg, \lambda) - f(\betasp, \lambda)) + \epsilonsp^\ast
    \end{align*}
    % Taking expectations on both sides and applying Lemma \ref{lma:ebError} yields the desired result, namely
    % \[
    % \frac{1}{2L}\bbE_{\lambdarv}\left[\|\nabla f(\bbeta,\lambdarv)\|^2\right] \leq F(\bbeta) - F(\betaavg) + \frac{L}{2}\epsilon_A^2
    % \]
    where the second inequality uses optimality of $\btheta^*(\lambda)$ and the third uses that $\epsilonsp^\ast$ is attained at $\betasp$.
    Taking expectations on both sides, we get
    \begin{align*}
        \frac{1}{2CL}\bbE\left[\|\nabla f(\bbeta,\lambdarv)\|^2\right] 
        \ \leq\ F(\bbeta) - F(\betaavg) + F(\betaavg) - F(\betasp) + \epsilonsp^\ast \ \leq\ F(\bbeta) - F(\betaavg) + \epsilonsp^\ast,
    \end{align*}
    where we used that the definition of $\betaavg$ implies that $F(\betaavg) \leq F(\betasp)$. Multiplying through by $2CL$ completes the proof.
\end{proof}

%%%%%%% PG: Commented out this old proof %%%%%%%
% \paul{OLD:}
%     Taking expectations on both sides, we get
%     \begin{align*}
%         \frac{1}{2CL}\bbE\left[\|\nabla f(\bbeta,\lambdarv)\|^2\right] 
%         \ &\leq\ F(\bbeta) - F(\betaavg) + \left(\Eb{h(\bPhi(\lambdarv)\betaavg, \lambdarv)} - \Eb{h(\btheta^*(\lambdarv), \lambdarv)}\right) 
%     \\
%         \ &=\ F(\bbeta) - F(\betaavg) + \left(\Eb{h(\bPhi(\lambdarv)\betaavg, \lambdarv)} - \Eb{h(\bPhi(\lambdarv)\betasp, \lambdarv)}\right) 
%     \\
%         & \qquad + \left(\Eb{h(\bPhi(\lambdarv)\betasp, \lambdarv) - h(\btheta^*(\lambdarv), \lambdarv)}\right).
%     \end{align*}
    
%      Notice that by definition of $\betaavg$, for any $\bbeta\in \bbR^p$ we have
%     \begin{equation}\label{ineq:lma3Midstep}
%         \Eb{h(\bPhi(\lambdarv)\betaavg, \lambdarv)} - \Eb{h(\bPhi(\lambdarv)\bbeta, \lambdarv)}\leq 0.
%     \end{equation}
    
%     As shown in the proof of \cref{thm:decompose}, $\epsilonsp^\ast$ is attained at some $\betasp\in\bbR^p$. Thus, 
%     \begin{align*}
%         \frac{1}{2CL}\bbE\left[\|\nabla f(\bbeta,\lambdarv)\|^2\right] 
%         \ &\leq\ F(\bbeta) - F(\betaavg) + \left(\Eb{h(\bPhi(\lambdarv)\betasp, \lambdarv) - h(\btheta^*(\lambdarv), \lambdarv)}\right)
%      \\
%         \ &\leq F(\bbeta) - F(\betaavg) + \epsilonsp^\ast.
%     \end{align*}
% \end{proof}

\begin{proof}[Proof for \Cref{thm:exact1}]
Our strategy will be to invoke \cref{lma:satisfyWGC} and apply \cref{lma:convergeWGC}, in conjunction with \cref{thm:decompose}.  To that end, notice that the upper bound on the step-size parameter $\bar \eta$ given in \cref{lma:convergeWGC} becomes $\frac{CL}{c\mu} \geq 1$.  Hence, we need only constrain $0 < \bar \eta \leq 1$.

Applying \cref{lma:convergeWGC} yields
    \[
        \Eb{\|\bbeta_T - \betaavg\|^2}
    \ \leq\ 
        \left(1 - \bar\eta \frac{c\mu}{CL}\right)^T\|\bbeta_0 - \betaavg\|^2 + \frac{2\bar \eta \epsilonsp^\ast}{c\mu}.
    \]
    Taking the expectation of \cref{thm:decompose} on both sides and substituting the above, we have 
    \begin{align}
        \Eb{\epsilonsp(\bbeta_T)}
    \ &\leq\ 
        2CL\Eb{\|\bbeta_T - \betaavg\|^2} + \frac{8CL}{c\mu}\epsilonsp^\ast \notag \\
    \ &\leq\ 
        2CL\left(1 - \bar \eta \frac{c\mu}{CL}\right)^T\|\bbeta_0 - \betaavg\|^2 
        + 4\frac{\bar \eta CL}{c \mu}\epsilonsp^\ast
        + \frac{8CL}{c\mu}\epsilonsp^\ast. \label{eqn:expected_bound_appendix}
    \end{align}
    Rearranging yields the first result.
Next, we prove the two bounds of the iteration complexity.  

First, consider the case $\epsilonsp^\ast > 0$.  Our goal will be to choose $T$ large enough to drive the first term on the right side of \eqref{eqn:expected_bound_appendix} to below $4\frac{CL}{c\mu} \bar\eta \epsilonsp^\ast$.  To simplify exposition, let $\kappa = \frac{CL}{c\mu}$.  Then,  we need to ensure that 
\[
    T \log(1-\bar\eta/\kappa) \leq \log\left(\frac{2\kappa \bar\eta \epsilonsp^\ast}{CL \| \bbeta_0 - \betaavg\|^2}\right).
\]
We can upper bound the right side using the identity $\log(1+x) \leq x$.  Hence, it would be sufficient if 
\[
T \geq \frac{\kappa}{\bar \eta  }\log\left(\frac{CL \| \bbeta_0 - \betaavg\|^2}{2\kappa \bar\eta \epsilonsp^\ast}\right).
\]
Replacing $\kappa$ with its definition yields the first case.  

% For the second case where $\epsilonsp^\ast = 0$, we follow a similar argument but seek to drive the error to below $\epsilon$.  We omit the details.

For the second case where $\epsilonsp^\ast = 0$, \eqref{eqn:expected_bound_appendix} reduces to
    \begin{align*}
        \Eb{\epsilonsp(\bbeta_T)}
    \ &\leq\ 
        2CL\left(1 - \bar \eta \frac{c\mu}{CL}\right)^T\|\bbeta_0 - \betaavg\|^2.
    \end{align*}
Given $\epsilon > 0$, we need to ensure that 
\[
    T \log(1-\bar\eta/\kappa) \leq \log\left(\frac{\epsilon}{2CL \| \bbeta_0 - \betaavg\|^2}\right).
\]
Again using $\log(1+x) \leq x$, it would be sufficient if 
\[
T \geq \frac{\kappa}{\bar \eta  }\log\left(\frac{2CL \| \bbeta_0 - \betaavg\|^2}{\epsilon}\right).
\]
Replacing $\kappa$ with its definition yields the second case.  

\end{proof}

\subsection{Proofs from \cref{sec:SpecializedResults}}

\begin{proof}[Proof of \cref{lem:CpcpLegendre}.]
As described in the main text, $\bPhi(\lambda) \in \mathbb R^{d \times qd}$ is block-diagonal with $d$ blocks of size $1 \times q$.  Denote this block by $\psi(\lambda) \in \mathbb R^{1 \times q}$ and note the elements of $\psi(\lambda)$ are precisely the first $q$ Legendre polynomials evaluated at $\lambda$.

Now the matrix $\Phi(\lambda)^\top \Phi(\lambda) \in \mathbb R^{qd \times qd}$ is also block diagonal with $d$ copies of the matrix $\psi(\lambda)^\top \psi(\lambda)$.  As a consequence, the eigenvectors of $\Phi(\lambda)^\top\Phi(\lambda)$ are the stacked copies of the eigenvectors of $\psi(\lambda)^\top\psi(\lambda)$ with the same eigenvalues.

Since $\psi(\lambda)^\top\psi(\lambda)$ is a rank-one matrix, it has at most one non-zero eigenvalue, and by inspection, this eigenvalue is $\| \psi(\lambda)\|^2$ with eigenvector $\psi(\lambda)^\top$.  Then, 
\begin{align*}
C_p = \sup_{\lambda \in[-1, 1]} \| \psi(\lambda)\|^2 = q,
\end{align*}
because the Legendre polynomials achieve their maxima at $1$ with a value of $1$.

By a nearly identical argument, we can see that 
\begin{align*}
c_p &= \sigma_{\min}\left\{ \Eb{\psi(\lambdarv)^\top \psi(\lambdarv)}\right\}
\\
& = \min_{n=0, \ldots, q-1} \frac{2}{2n + 1} 
\\ 
&= 
\frac{2}{2q -1},
\end{align*}
because by the orthogonality of the Legendre polynomials, the above matrix is diagonal.  Hence \[C_p/c_p \ =\ \frac{2q^2- q}{2} \leq q^2.\]
\end{proof}

\begin{proof}[Proof for \cref{lma:chebyshev}]
Our proof is constructive and we will show a slightly stronger result.  We will approximate $\btheta^*_i(\lambda)$ by its Chebyshev truncation up to degree $q$ for each $i$.  Letting $\bar\bbeta$ be the coefficients corresponding to the resulting polynomials, we will show that $\epsilonsp(\bar\bbeta)$ satisfies the bound in \cref{lma:chebyshev}, which implies that $\epsilonsp^\ast$ also satisfies the bound.

% To see this result, we first show that this decay rate holds for the Chebyshev polynomial basis using a result from \citep{trefethen2009approximation}. Since polynomial bases of the same degree are equivalent to each other up to coefficients, the best-in-class interpolation error of Chebyshev polynomials can stand for all polynomial bases.

Let $T_n(\lambda)$ denote the $n^\text{th}$ Chebyshev polynomial of the first kind. For the $i^\text{th}$ dimension of the solution path, let $a_{i,n}$ denote the coefficient of $T_n$ in the Chebyshev truncation of $\btheta_i^\ast$ up to degree $q$.

% Since $\btheta^{(\nu)*}_i(\lambda)$ is Lipschitz continuous with constant $V$, it also has bounded variation $V$. 
We assumed that $\btheta^{*}_i(\lambda)$, ..., $\btheta^{(\nu-1)*}_i(\lambda)$ are absolutely continuous and $\btheta^{(\nu)*}_i(\lambda)$ has bounded variation $V$. When $\nu = 0$, we are assuming simply that $\btheta^{*}_i(\lambda)$ is of bounded variation $V$.
Thus, for any $q \geq \nu + 1$ and $i\in [d]$, \citet[Theorem 7.2]{trefethen2009approximation} guarantees that 
\[
    \sup_{\lambda\in [-1, 1]} \left|\sum_{n=0}^q a_{i,n}T_n(\lambda) - \btheta_i^\ast(\lambda)\right|
~\leq~ 
    \frac{2V}{\pi \nu(q-\nu)^\nu}.
\]
Then,
\[
    \epsilonsp(\bar\bbeta)
~ = ~
    \sup_{\lambda\in \Lambda}\left\{h\left(\left(\sum_{n=0}^q a_{i,n}T_n(\lambda)\right)_{i=[d]}, \lambda\right) - h(\btheta^*(\lambda), \lambda)\right\}.
\]
Moreover, smoothness of $h(\cdot, \lambda)$ ensures that for any $\lambda \in \Lambda$,
\[
    h\left(\left(\sum_{n=0}^q a_{i,n}T_n(\lambda)\right)_{i=[d]}, \lambda\right) - h(\btheta^*(\lambda), \lambda)
~\leq~
    \frac{L}{2}\left\|\left(\sum_{n=0}^q a_{i,n}T_n(\lambda)\right)_{i=[d]} - \btheta^*(\lambda)\right\|^2.
\]
Combine the above and using the definition of $l_2$-norm, we deduce that 
\begin{align*}
    \epsilonsp(\bar\bbeta)
&~\leq~
    \sup_{\lambda\in \Lambda}\left\{\frac{L}{2}\left\|\left(\sum_{n=0}^q a_{i,n}T_n(\lambda)\right)_{i=[d]} - \btheta^*(\lambda)\right\|^2\right\}\\
&~\leq~
    \sup_{\lambda\in [-1, 1]}\left\{\frac{dL}{2}\cdot \sup_{i\in [d]} \left|\sum_{n=0}^q a_{i,n}T_n(\lambda) - \btheta_i^\ast(\lambda)\right|^2\right\}\\
&~\leq~
    \frac{dL}{2}\left(\frac{2V}{\pi \nu(q-\nu)^\nu}\right)^2.
\end{align*}

We note that this result is presented and holds for any $\nu \geq 0$.  (When $\nu = 0$, the bound is infinity and hence trivially valid.)  As will be seen in the proof of \cref{thm:SGDnuDifferentiable} below, however, we will only apply the result when $\nu \geq 2$.

\end{proof}

%%%%%%%%%%%%%%%%%%%%%%
%%%%%%%%%%%%%%%%%%%%%%
%%%%%%%%%%%%%%%%%%%%%%
\begin{proof}[Proof of  \cref{thm:SGDnuDifferentiable}]
As we are only interested in asymptotic behavior as $\epsilon\rightarrow 0$, we will often suppress any constants that do not depend on $\epsilon$ below.  In particular, we will write $a \lsim b$ whenever there exists a constant $C$ (not depending on $\epsilon$ but perhaps depending on $h$ and $\Lambda$) such that $a \leq Cb$.

Note that in order to attain a small solution path error, we will need to use a large number of polynomials $q$. Based on \cref{thm:exact1}, our first goal is to choose $q$ large enough that
\begin{align} \label{goal_q}
\epsilon \ \geq \ \frac{8L(\bar \eta + 1)}{\mu} \frac{C_q}{c_q} \cdot \epsilonsp^\ast(q).
\end{align}

First observe that \cref{lma:chebyshev} establishes that the solution path error is at most 
\begin{align} \label{eq:AsymptoticBoundEps_nudiff}
    \epsilonsp^\ast(q) \ \leq \ \frac{ 4 dL V^2}{2 \pi^2 \nu^2 (q - \nu)^{2\nu}} \ \lsim q^{-2 \nu}.
\end{align}
Using \cref{lem:CpcpLegendre}, we can upper bound the right side of \cref{goal_q} by 
\[
  \frac{8L(\bar \eta + 1)}{\mu} \frac{C_q}{c_q} \cdot \epsilonsp^\ast(q)
  \lsim q^{2(1-\nu)}.
\]
Hence it suffices to take 
\[
    %\frac{\epsilon}{A_1}  \geq  q^{2(1-\nu)} \iff 
    \epsilon^{-\frac{1}{2 (\nu -1)}} \lsim q
\]
polynomials to achieve our target error.  

We now seek to bound the iteration count. By \cref{thm:exact1}, the iteration count should exceed
\begin{align*}
\frac{L}{\bar \eta \mu} \frac{C_q}{c_q} \log\left( \frac{c_q \mu \| \bbeta^*_{\rm avg}\|^2}{2 \bar \eta \epsilonsp^\ast}\right)
\ \lsim \
q^2 \left(\log \| \betaavg \|^2 + \log(q ) \right)
\end{align*}
where we have assumed SGD was initialized at $\bbeta = 0$ and used \cref{eq:AsymptoticBoundEps_nudiff} and  $c_q = \frac{2}{2q -1}$ (cf. the proof of \cref{lem:CpcpLegendre}) to simplify.

To complete the proof we need to bound $\| \betaavg \|^2$.  To this end, we first bound $\| \bar \bbeta \|$, where $\bar\bbeta$ is constructed by approximating  each component $\theta_i^*(\lambda)$ by its Chebyshev truncation to degree $q$ for each $i$.

Then,
\begin{align*}
\| \bar \bbeta \|^2 &= \sum_{i=1}^d \sum_{k=1}^q \abs{a_{ik}}^2 
\\
& = \sum_{i=1}^d \sum_{k=1}^\nu \abs{a_{ik}}^2 + \sum_{i=1}^d \sum_{k=\nu+1}^q \abs{a_{ik}}^2.
\\
& \lsim 
\sum_{k=\nu+1}^q \frac{1}{k^{2(\nu + 1)}},
\\
& \lsim 
\sum_{k=\nu+1}^\infty \frac{1}{k^{2(\nu + 1)}},
\end{align*} 
where the penultimate inequality collects constants and invokes \citet[Theorem 7.1]{trefethen2009approximation} to bound the second summation. Note, for $\nu \geq 0$, this last sum is summable, so that $\|\bar \bbeta\| \lsim 1$.

Furthermore, from the proof  of \cref{lma:chebyshev}, $\epsilonsp(\bar\bbeta) \lsim q^{-2\nu}$.  Hence, since the solution-path error upper bounds the stochastic error, we also have $F(\bar\bbeta) - F(\betaavg) \lsim q^{-2\nu}$.  From strong convexity, this implies that $\| \bar \bbeta- \betaavg \|  \ \lsim \ q^{-\nu}$.

Putting it together, we have that 
\[
    \| \betaavg\| \ \leq \ \| \betaavg - \bar \bbeta \| + \| \bar\bbeta \|\ \lsim\ q^{-\nu} + 1\ \lsim\ 1.
\]

Substituting above shows that it suffices to take $O(q^2 \log q)$ iterations as $\epsilon \rightarrow 0$.  Given our previous calculation of $q$, this amount to $O(\bar\epsilon^{\frac{1}{1-\nu}} \log(1/\bar\epsilon))$ iterations.

\end{proof}

%%%%%%%%%%%%%%%%%%%%%%
%%%%%%%%%%%%%%%%%%%%%%
%%%%%%%%%%%%%%%%%%%%%%

\begin{proof}[Proof for \cref{lma:analytic}]
The proof is very similar to the proof of \cref{lma:chebyshev}, with the only difference being that we use Theorem 8.2 from \citep{trefethen2009approximation} instead of Theorem 7.2. 

Recall that a Bernstein ellipse $\mathcal E_\omega$ with radius $\omega$ in the complex plain is the ellipse
\[
    \mathcal E_\omega = \left\{z = \frac{1}{2}\left(\omega e^{\theta\sqrt{-1}} + \omega^{-1} e^{-\theta\sqrt{-1}}\right) : \theta \in [0, 2\pi)\right\}.
\]
Since $\theta^*_i(\lambda)$ is analytic, there exists an analytic continuation of $\theta^*_i(\lambda)$ to a Bernstein ellipse of radius $\omega_i > 1$.  The value of $\omega_i$ generally depends on if $\theta^*_i(\lambda)$ has any singularities in the complex plane, but is guaranteed to be larger than $1$.
%%VG:  This is a bit of a bold conjecture from me based on my interpretation of this tack exchange:
%https://math.stackexchange.com/questions/3239654/extension-of-function-analytic-on-1-1
%Notice the closed interval is necessary.  There are counter examples if you use the open interval.
  Let $M_i$ be $\max_{z \in \mathcal E_{\omega_i}} \abs{\theta^*_i(z)}$.  Finally, let $\omega \equiv \min_{i=1, \ldots d} \omega_i > 1$ and $M = \max_{i=1, \ldots, d} M_i < \infty$. 

Now, define $T_n(\lambda)$ and $a_{i,n}$ as in the proof for \cref{lma:chebyshev}.

For any $q > 0$ and $i\in [d]$, Theorem 8.2 in \citep{trefethen2009approximation} guarantees that 
\[
    \sup_{\lambda\in [-1, 1]} \left|\sum_{n=0}^q a_{i,n}T_n(\lambda) - \btheta_i^\ast(\lambda)\right|
~\leq~ 
    \frac{2M\omega^{-q}}{\omega - 1}.
\]
Plug this result into the last inequality of the proof for \cref{lma:chebyshev}, we obtain
\begin{align*}
    \epsilonsp(\bar\bbeta)
~\leq~
    \sup_{\lambda\in [-1, 1]}\left\{\frac{dL}{2}\cdot \sup_{i\in [d]} \left|\sum_{n=0}^q a_{i,n}T_n(\lambda) - \btheta_i^\ast(\lambda)\right|^2\right\}
~\leq~
    \frac{dL}{2}\left(\frac{2M\omega^{-q}}{\omega - 1}\right)^2.
\end{align*}

\end{proof}

\begin{proof}[Proof of \cref{thm:SGDforAnalytic}.]
The proof is quite similar to the proof of \cref{thm:SGDnuDifferentiable}.  Again, we only consider asymptotic behavior as $\epsilon \rightarrow 0$ and suppress all other constants. Hence, $ a\lsim b$ means there exists a constant $C$ (not depending on $\epsilon$) such that $ a \leq Cb$. 

Again, our first goal is to identify a $q$ sufficiently large to achieve our target error.  By \cref{thm:exact1} we seek $q$ large enough that 
\[
    \epsilon \geq \frac{8L(\bar \eta + 1)}{\mu} \frac{C_q}{c_q} \epsilonsp^*(q).
\]
\Cref{lma:analytic} shows that there exists an $\omega > 1$ such that
\[
    \epsilonsp^\ast(q) \ \lsim \ \omega^{-2q},
\]
hence it suffice to take $q$ large enough that 
\[
    q^2 \omega^{-2q} \lsim \bar\epsilon.
\]
Solving this equation exactly requires the Lambert-W function.  Instead, we take $q = \frac{\log(1/\bar\epsilon)}{\log \omega}$.  Then, 
\[
    q^2 \omega^{-2q} = \frac{\epsilon ^2 \log ^2\left(\frac{1}{\epsilon }\right)}{\log ^2(\omega)} \lsim \bar\epsilon
\]
for $\bar\epsilon$ sufficiently small.  

Again, to bound the number of iterations, we will need to bound $\| \betaavg - \bbeta_0\|$.  We again will assume that $\bbeta_0 = \mathbf 0$, and first consider bounding $\bar \bbeta$.  Recall, $\bar\bbeta$ is obtained by approximating each component $i$ by the Chebyshev truncation of $\btheta^*(\lambda)$ to degree $q$. Then, 
\begin{align*}
\| \bar\bbeta\|^2  &= \sum_{i=1}^d \sum_{k=1}^q \abs{a_{ik}}^2 
\\ 
& \ \lsim \ 
\sum_{k=1}^q \omega^{-k}
\\
& \ \leq \sum_{k=1}^\infty \omega^{-k}.
\end{align*}
Here, the second to last inequality uses \cite[Theorem 8.1]{trefethen2009approximation}.  Since the last summation is summable, we again conclude that $\| \bar \bbeta \|_2 \lsim 1$.

The proof of \cref{lma:analytic} establishes that $\epsilonsp(\bar\bbeta)  \lsim \omega^{-2q}$.  Since solution path error upper-bounds the stochastic error, we conclude that $F(\bar \bbeta) - F(\betaavg) \lsim \omega^{-2q}$, and by strong convexity, $\| \bar\bbeta - \betaavg\| \lsim \omega^{-q}$.

Putting it together, we have that
\[
    \| \betaavg\| \leq \ \|\betaavg - \bar\bbeta\| + \|\bar\bbeta \|\ \lsim\ \omega^{-q} + 1\ \lsim\ 1.
\]

Substituting into the iteration complexity shows that it suffices to take
\begin{align*}
q^2 \left(\log \| \betaavg \| + \log q\right)
& \lsim 
\log^2(1/\bar\epsilon) \log\log(1/\bar\epsilon)
\end{align*}
iterations, by using our condition on $q$.
\end{proof}

\section{Implementation Details}
\subsection{Calibrating Uniform Discretization}\label{sec:fixScheduleNGS}

To apply uniform discretization in practice one must decide i) the number of grid points to use, and ii) the number of gradient calls to make at each grid point. Loosely speaking, our approach to these two decisions is to first fix a set of desired ``target" solution path errors denoted $\Delta$.  Then, for each $\delta \in \Delta$, we determine the number of grid points and gradient calls to approximately achieve a solution path error of $\delta$.

More specifically, \citet{ndiaye2019safe}
suggests that to achieve a solution path error of $\delta$, we require $O(1/\sqrt{\delta})$ grid points.  %We adopt this guideline for both the one-dimensional example and the two-dimensional example that we consider. For each experiment, we first fix a list of target solution path errors, denoted as $\Delta$. 
Thus, for each $\delta \in \Delta$, we construct a uniform discretization with $1/\sqrt{\delta}$ grid points across every dimension of the hyperparameter space. We denote this grid by $G(\delta)$.  Recall that deterministic gradient descent requires $O(\log(1/\delta))$ steps to achieve an error of $O(\delta)$.  Motivated by this result, we use $c_1 \log(c_2/\delta)$ gradient calls per grid point.  The total number of gradient calls is thus 
 $c_1 \log(c_2/\delta)/\sqrt{\delta}$.

% The decision of (ii) relies on the optimization method we use at each grid. For the weighted binary classification problem, we employ deterministic gradient descent, which converges in $O(\log(1/\delta))$ gradient calls per grid point. For the portfolio allocation problem, we use BFGS, which has a convergence rate of $O(\log^{1/(1+\alpha)}(1/\delta))$. In both cases, we heuristically execute the optimization algorithm for $c_1 \log(c_2/\delta)$ gradient calls.  

The specific values of $c_1$, $c_2$, and $\Delta$ used in our experiments are recorded in \cref{tbl:parameters_for_numerics}.
\begin{table}[hbt!]
\centering
\begin{center}
\begin{tabular}{ |c|c|c|c| } 
 \hline
Weighted Binary Classification & $c_1 = 1$ & $c_2 = 0.5$ & $\Delta = \{2^{-6}, 2^{-6.5}, 2^{-7}, \ldots, 2^{-18}\}$ \\ 
\hline
 % Weighted Binary Classification ($w(\lambda) = \lambda/(1+\lambda)$) & $c_1 = 2$ & $c_2 = 1$ & $\Delta = \{2^{-6}, 2^{-6.5}, 2^{-7}, \ldots, 2^{-18}\}$ \\ 
 % \hline
 Portfolio Allocation & $c_1 = 0.65$ & $c_2 = 1$ & $\Delta = \{4^{-2}, 5^{-2}, \ldots, 19^{-2}\}$ \\ 
 \hline
\end{tabular}
\caption{{\bf Parameters Used in Experiments.}}
\label{tbl:parameters_for_numerics}
\end{center}
\end{table}

We next argue that this procedure with these constants is roughly efficient. Define the \emph{grid pass error} as 
\[
\sup_{\lambda \in G(\delta)} h(\hat{\btheta}(\lambda), \lambda)- h(\btheta^*(\lambda), \lambda).
\]
For comparison, the solution-path error is
\[
\epsilonsp(\hat{\btheta}(\cdot)) = \sup_{\lambda \in \Lambda} h(\hat{\btheta}(\lambda), \lambda)- h(\btheta^*(\lambda), \lambda).
\]
Since $G(\delta) \subseteq \Lambda$, the grid pass error is always less than the solution path error.

\begin{figure}[hbt!]
     \centering
     \begin{subfigure}[b]{0.43\textwidth}
         \centering
         \raisebox{0cm}{
         \includegraphics[width=\textwidth]{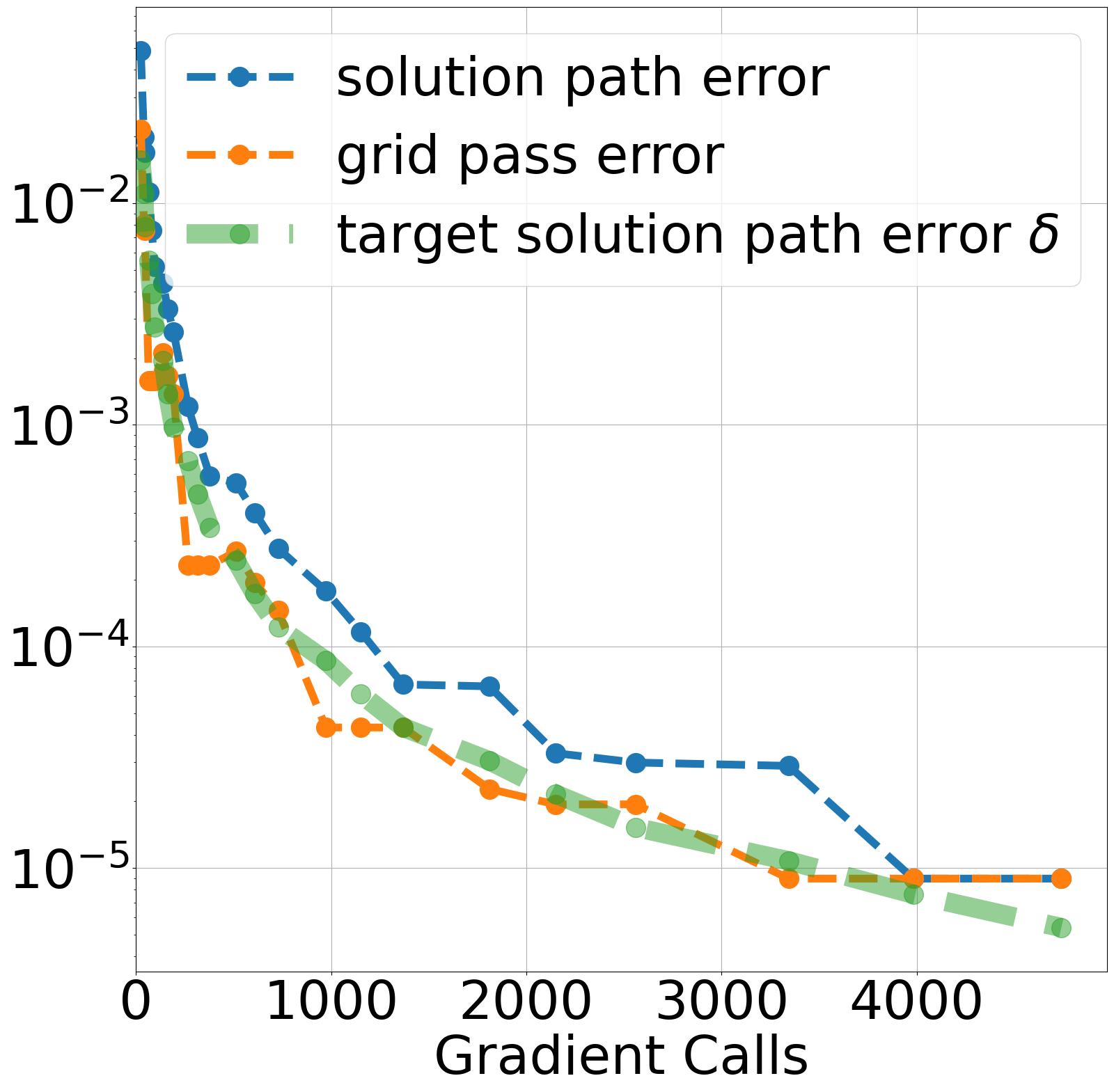}}
         \caption{}
         \label{fig:NGS_stopping_criterion_exact_legendre}
     \end{subfigure}
     \hfill
     \begin{subfigure}[b]{0.47\textwidth}
        % \vspace{1.5cm}
         \centering
         \includegraphics[width=\textwidth]{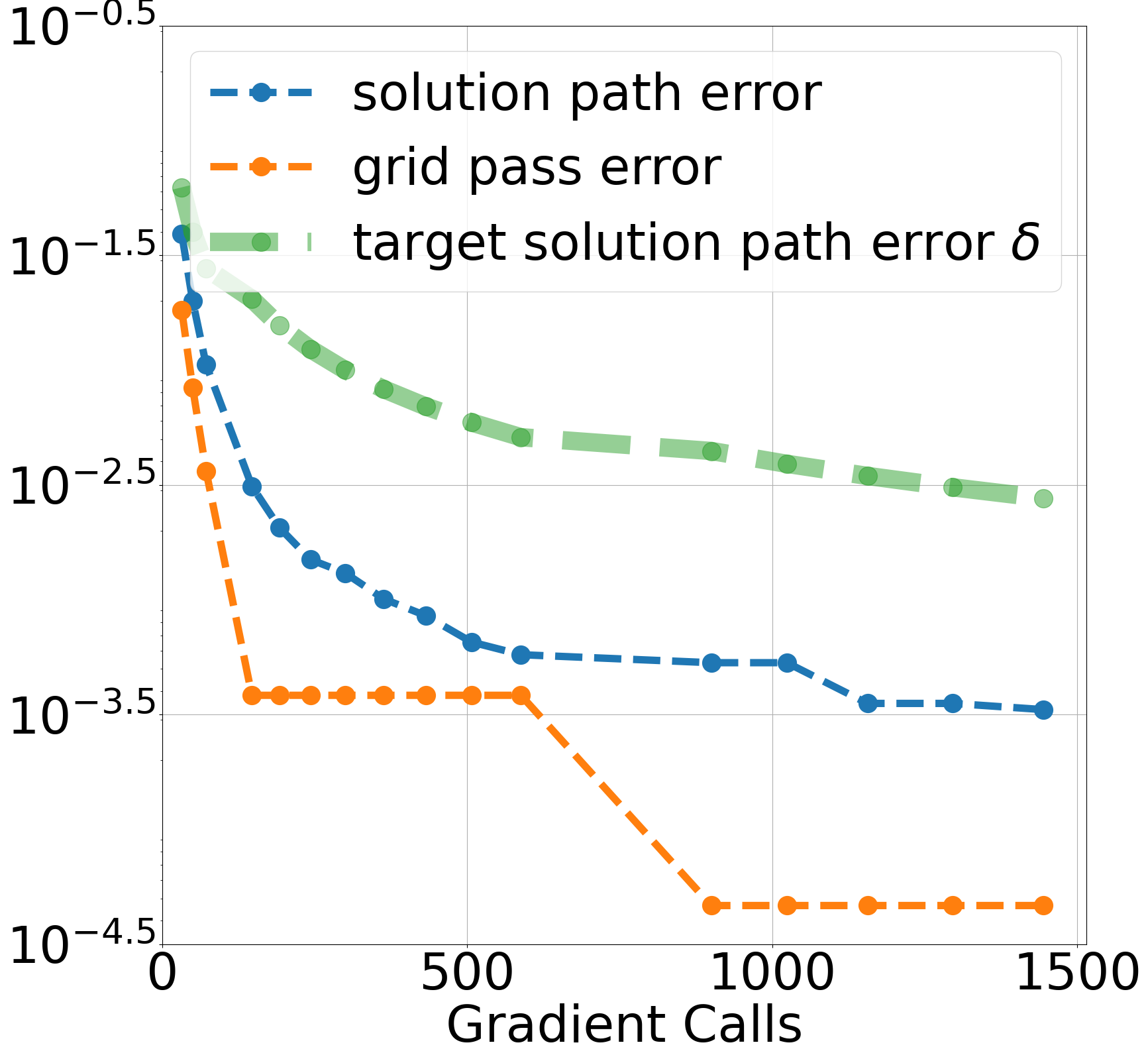}
         \caption{}
         \label{fig:NGS_stopping_criterion_exact_2d}
     \end{subfigure}
        \caption{{\bf Compare: Ideal Stopping Criterion, Grid Pass Error, and Solution Path Error.} (a) weighted binary classification; (b) portfolio allocation. Solution path error and grid pass error are plotted against the corresponding total number of gradient calls incurred in practice. Target solution path error $\delta$ is plotted against $c_1\log(c_2/\delta)/\sqrt{\delta}$, which also equals the total number of gradient calls incurred in practice. Each discrete point in the line plots above corresponds to a discretization grid point $\delta \in G(\delta)$; thus smaller $\delta$ is associated with more gradient calls.}
        % \label{fig:three graphs}
        \label{fig:stoppingCriterionNGS}
\end{figure}

By comparing the grid pass error, the solution path error, and the target solution path error, we can calibrate the amount of work done at each grid point and the total number of grid points.  Specifically, if the grid pass error is much smaller than the solution pass error, it suggests we have allocated too much work to gradient calls at grid points and have insufficient grid points.   On the other hand, if the grid pass error is very far from the target solution path error, it suggests we have not allocated enough work to gradient calls at the grid points and have too many grid points.  Based on this trade-off, we tune the  constants $c_1$ and $c_2$.

% are tuned carefully to achieve two main objectives:
% \begin{enumerate}
%     \item the grid pass error should be close to the solution path error, which implies that we are not doing an excessive amount of computations at each grid;
%     \item both the grid pass error and the solution path error should be close to the target solution path error $\delta$, which indicates that we are performing sufficient computations at each grid.
% \end{enumerate}

\cref{fig:stoppingCriterionNGS} compares the solution path error and grid pass error to the target solution path error for our experiments. \cref{fig:NGS_stopping_criterion_exact_legendre} illustrates that in the weighted binary classification experiment, the discretization scheme performs well, satisfying both objectives as the solution path error and grid pass error closely align with the target solution path error.  \cref{fig:NGS_stopping_criterion_exact_2d} demonstrates that for the portfolio allocation problem, both solution path error and grid pass error are still close to each other, but fall below the target solution path error. This is primarily due to the high precision of L-BFGS, though only very few gradient calls are performed at each grid.

\subsection{Basis for Moderate Dimensional Portfolio Allocation}\label{sec:taylorExpansion}

Recall that this experiment uses the following objective function: 
\[
h(\theta, \lambda) \ = \ 
-\lambda_1 \cdot \mu^\top \theta + \lambda_2 \cdot \theta^\top \Sigma \theta + \|\theta - \lambda_{3:12}\|_2^2 ,
\]
where $\lambda_{3:12}$ represents the current holdings.

Since this is a quadratic objective, let us first directly compute the optimal solution path
\[
\btheta^*(\lambda) \ = \ 
\frac{1}{2} \left( \bm I + \lambda_2 \Sigma \right)^{-1} \left( \lambda_1 \mu + 2 \lambda_{3:12} \right).
\]

Consider the eigendecomposition of $\Sigma$, which takes the form \(\Sigma = P^\top D P\) for some orthonormal matrix $P$ and diagonal matrix $D = \text{diag}(d_1, \ldots, d_{10})$.
Plugging into the above closed form optimal solution path yields
\[
\btheta^*(\lambda) \ = \ 
\frac{1}{2} P^\top 
\text{diag}\left(\frac{1}{1 + \lambda_2 d_1}, \ldots, \frac{1}{1+ \lambda_2 d_{10}} \right) 
P \left( \lambda_1 \mu + 2 \lambda_{3:12} \right).
\]

Using the Taylor series expansion for $\frac{1}{1 + a \lambda_2}$ around $0$
\[
\frac{1}{1 + a \lambda_2} =  \sum_{i=0}^\infty (-1)^i (a\lambda_2)^i, \qquad a\in \bbR,
\]
we conclude that, as $q\to \infty$, each component of the optimal solution path $\btheta^*_i(\lambda)$ is in the span of our chosen basis:
\[
\lambda_{(j\bmod 12)} \cdot \lambda_2^{\lfloor{j / 12}\rfloor},
\qquad
 j = 1, ..., q.
\]

\section{Additional Figures}
\label{sec:AdditionalFigures}
\begin{figure}[H]
     \centering
     \begin{subfigure}[b]{0.46\textwidth}
         \centering
         \raisebox{0cm}{
         \includegraphics[width=\textwidth]{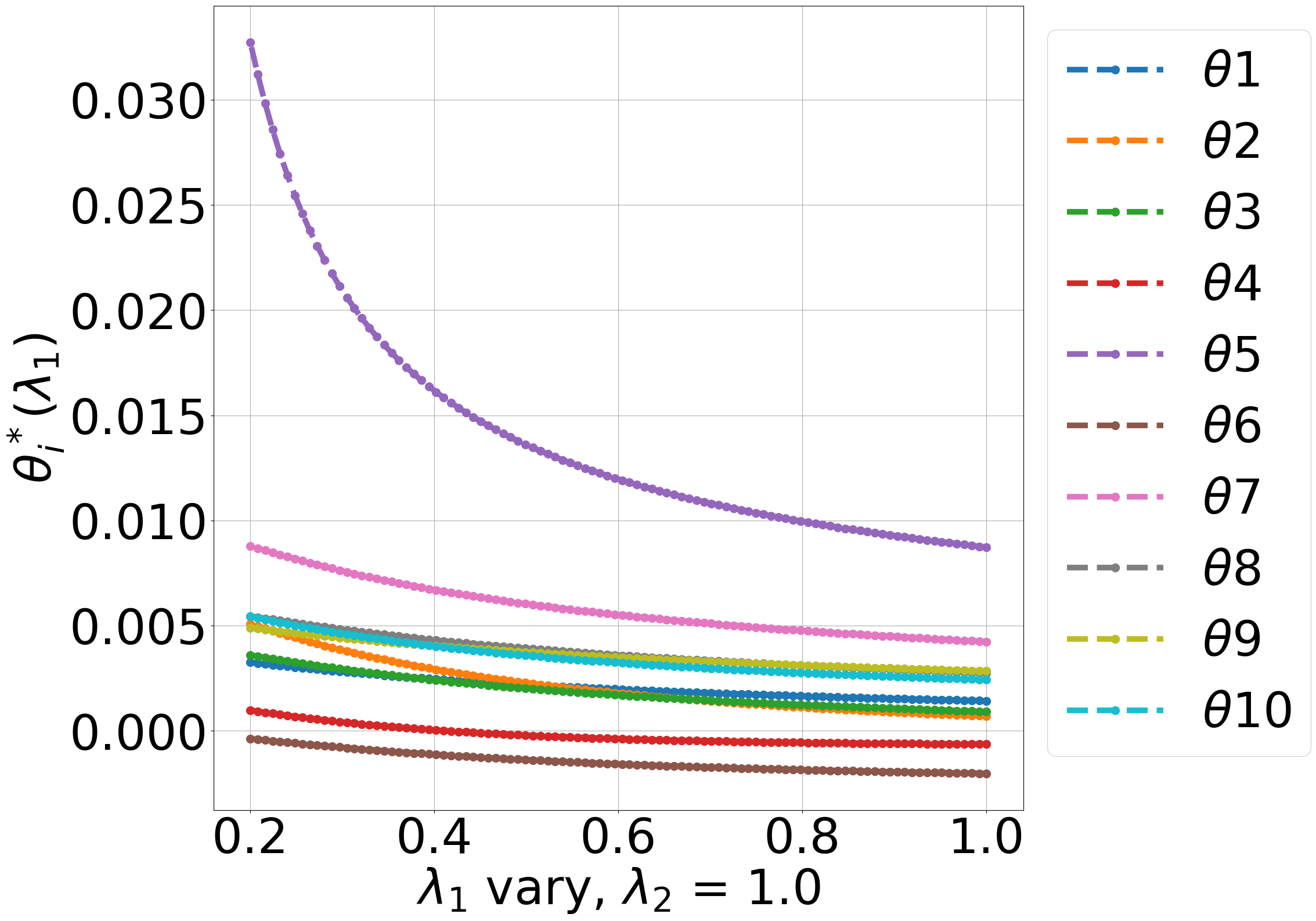}}
         % \caption{}
     \end{subfigure}
\hfill
     \begin{subfigure}[b]{0.44\textwidth}
        % \vspace{1.5cm}
         \centering
         \includegraphics[width=\textwidth]{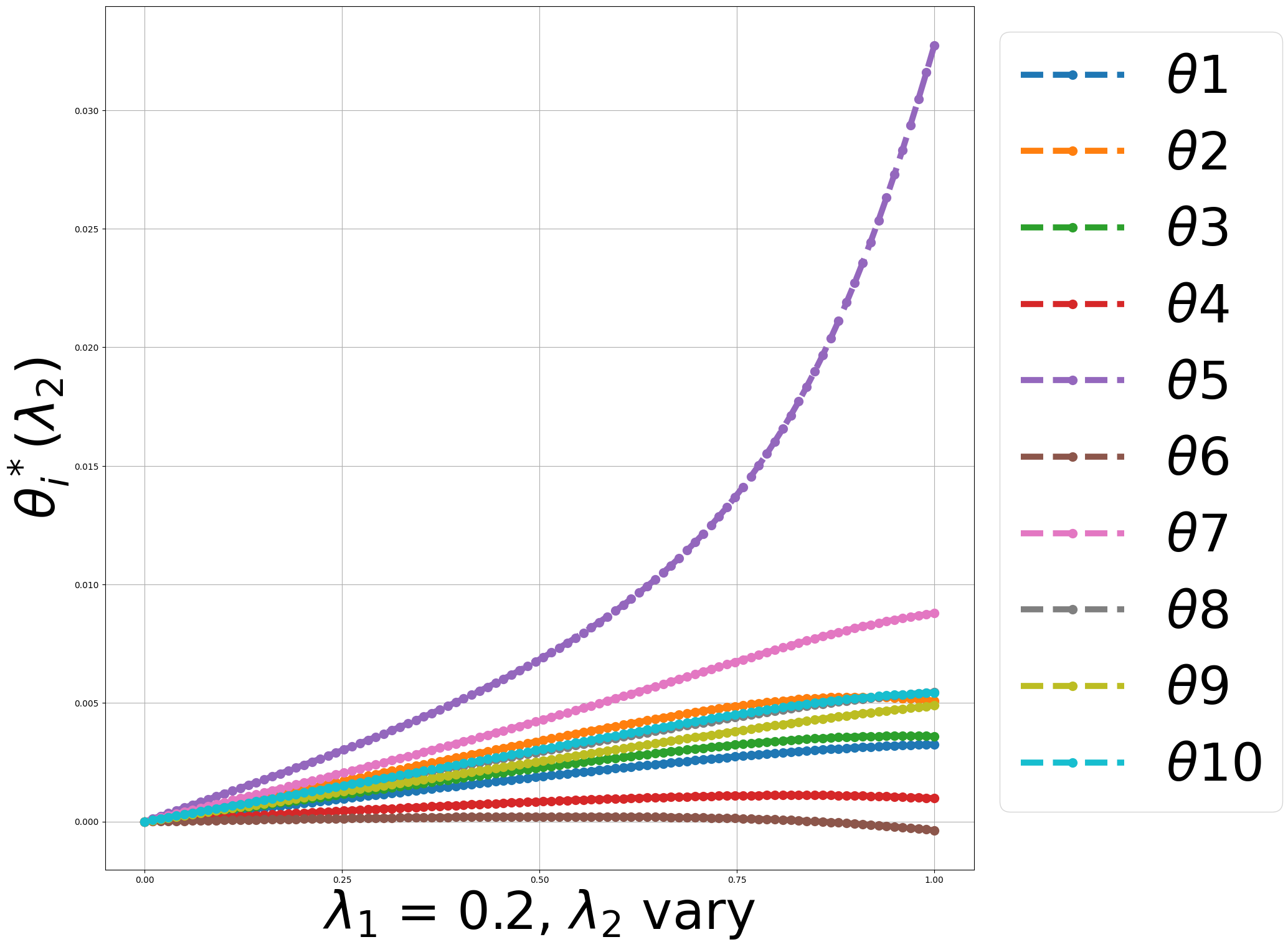}
         % \caption{}
     \end{subfigure}
        \caption{{\bf $\btheta^*(\lambda)$ for Portfolio Allocation from \cref{sec:portfolioAlloc}.} The true solution paths are very smooth in each dimension, suggesting that this problem is highly interpolable by a polynomial basis.}
        % \label{fig:three graphs}
        \label{fig:true_soln_path_2d}
\end{figure}

\begin{figure}[hbt!] 
\begin{center}
\includegraphics[width=.7\textwidth]{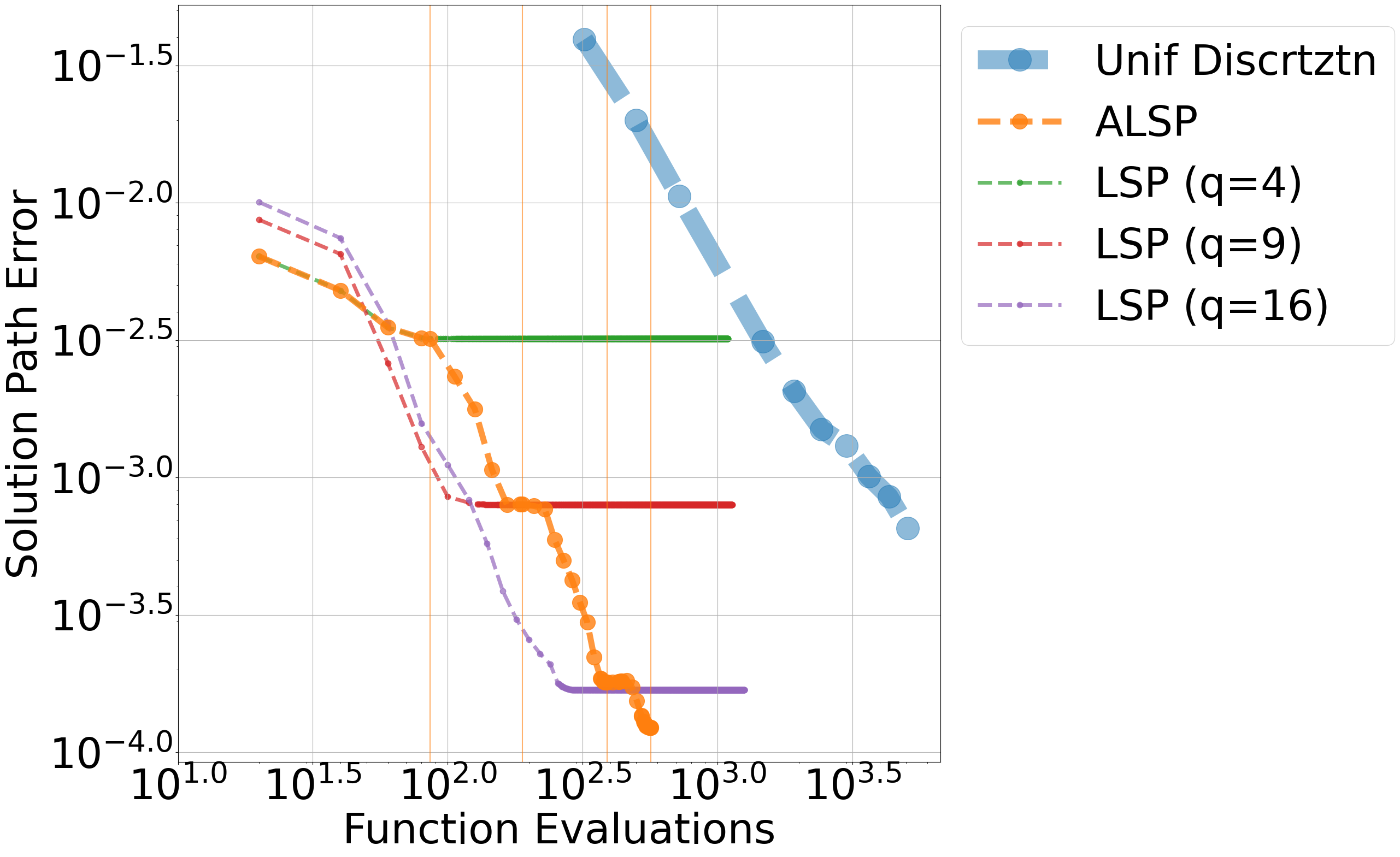}
\end{center}
\caption{{\bf LSP and ALSP for Portfolio Allocation.} Compares methods using the first $q = 4, 9, 16$ bivariate-Legendre polynomials. Differs from \cref{fig:boosted_and_fixed_exact_2d} as the $y$-axis records the number of function evaluations during line-search of L-BFGS instead of gradient calls.
}
\end{figure}

\begin{figure}[hbt!] 
\begin{center}
\includegraphics[width=.7\textwidth]{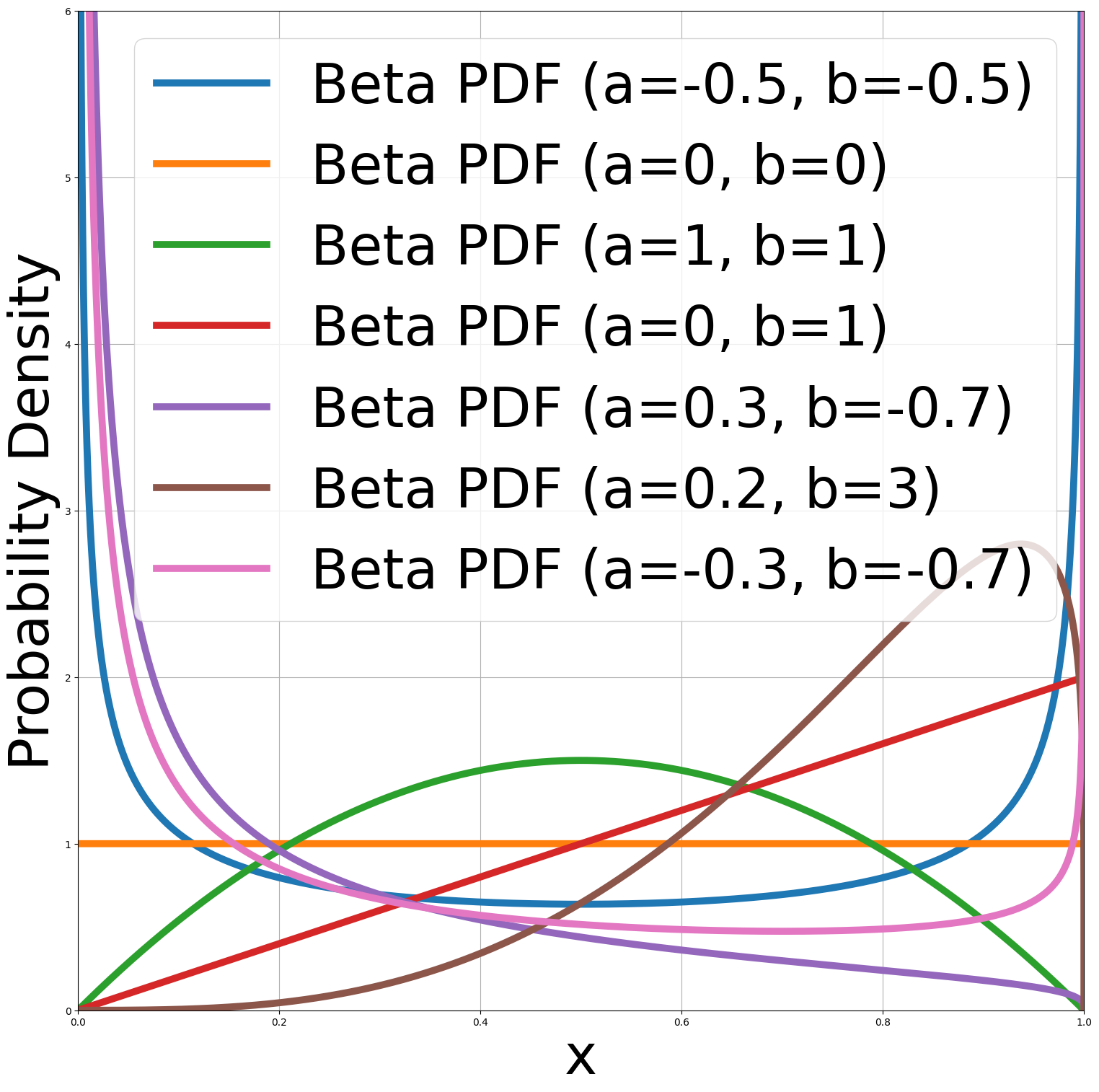}
\end{center}
\caption{{\bf Probability Density Function(PDF) of Beta Distribution over $[0, 1]$ with Various $a$, $b$.}
}
\end{figure}

\end{document}